\documentclass[12pt]{amsart}

\usepackage{MyStyle}

\begin{document}

\title{Spectrum of singularities, analytic torsion and Calabi--Yau degenerations}
\author{Dennis Eriksson}
\author{Gerard Freixas i Montplet}

\address{Dennis Eriksson \\ Department of Mathematics \\ Chalmers University of Technology and  University of Gothenburg}
\email{dener@chalmers.se}

\address{Gerard Freixas i Montplet \\ CNRS -- Centre de Math\'ematiques Laurent Schwartz - \'Ecole Polytechnique - Institut Polytechnique de Paris}
\email{gerard.freixas@polytechnique.edu}


\begin{abstract}
    We relate the spectrum of singularities to the geometry of degenerating Calabi--Yau manifolds through holomorphic analytic torsion. We express the leading asymptotic coefficients of the analytic torsion of holomorphic differential forms in terms of Hodge-theoretic invariants of vanishing cycles. For isolated hypersurface singularities, these formulas lead to the formulation of new higher Durfee--Saito conjectures bounding Hodge-filtration dimensions by the Milnor number and Eulerian numbers. We also prove that, for isolated quasi-homogeneous singularities, the spectral measure is dominated in convex order by the corresponding Irwin--Hall distribution. Applying the torsion formulas to the BCOV invariant yields a computable obstruction to birational smooth fillings. For isolated singularities, its local contribution is expressed in terms of Hertling's spectral variance and a Bernoulli sum determined by monodromy. As applications, we rule out birational smooth fillings for projective Calabi--Yau degenerations with smooth total space and singular central fiber having only ADE or terminal Brieskorn--Pham singularities, in relative dimension at least three, extending results of Voisin.
\end{abstract}

\maketitle

\setcounter{tocdepth}{2}
\tableofcontents

\section{Introduction}

\subsection{Friedman's problem and the BCOV obstruction}
\label{subsection:leitmotif}

A classical problem in algebraic geometry is to determine when a projective degeneration admits a smooth filling. For abelian varieties, the Néron--Ogg--Shafarevich criterion gives a positive answer in terms of the finiteness of monodromy, a result which has been extended to hyperkähler varieties \cite{KLSV-HK}. Friedman emphasized this same problem for other degenerations \cite{Friedmansimthreefolddouble,Friedman:degenerating-family}. For general Calabi--Yau varieties \footnote{By this we mean a compact K\"ahler manifold with trivial canonical bundle.}, finite monodromy does not suffice, as shown by several examples \cite{CynkStraten,NicaiseLunardon,Voisin:fillings,wang}. Even for Lefschetz pencils, the problem leads to subtle Hodge-theoretic questions, as demonstrated in \cite{Voisin:fillings}. Existing results in the Calabi--Yau setting are nevertheless largely case-specific.

\subsubsection{} The first aim of this paper is to provide a systematic numerical obstruction to birational smooth fillings of Calabi--Yau degenerations. The obstruction is extracted from the BCOV invariant, originating in mirror symmetry in theoretical physics \cite{bcov}. It was constructed mathematically by Fang, Lu and Yoshikawa \cite{FLY} for Calabi--Yau threefolds and, in arbitrary dimension, in our work with Mourougane \cite{cdg2}.

\subsubsection{} To define our invariant, let $f\colon X\longrightarrow \DBbb$ be a projective degeneration of Calabi--Yau manifolds, smooth over the punctured disc. For the asymptotic statements with non-isolated singularities, we assume that the family is the restriction of a morphism over an algebraic curve. In \cite{cdg2} we showed that, allowing the central fiber $X_0$ to be singular, the BCOV invariant satisfies
\begin{equation}\label{eq:kappa-intro}
\log \tau_{\mathrm{BCOV}}(X_t)
=
\kappa_f\cdot\log |t|^2+o(\log |t|),
\end{equation}
as $t\to 0$, for a rational number $\kappa_f$ determined by $f_{\mid\DBbb^{\times}}$.

\subsubsection{} The smooth filling problem asks whether $f_{\mid \DBbb^{\times}}$ extends smoothly across the origin. Non-trivial monodromy, even finite, rules out such a filling. We hence allow a further finite base change to eliminate finite monodromy obstructions. In the case of a smooth filling with $X_0$ not necessarily projective but Kähler, we prove in Section \ref{sec:obstructionsmoothfillings} that the left hand side of \eqref{eq:kappa-intro} extends continuously across the origin, and hence that $\kappa_f=0$.  More strongly, $\kappa_f\neq 0$ obstructs not only smooth fillings in Friedman's sense, but even birational smooth fillings, by the birational invariance of $\tau_{\mathrm{BCOV}}$ proven in \cite{ZhangFu-birationalBCOV}.

\subsubsection{}
Our approach therefore leads to a numerical problem: compute $\kappa_f$ and determine whether it vanishes. Previous formulas for $\kappa_f$ were restricted to special geometries or normal-crossings models \cite{cdg2,FLY}, making them difficult to compute. The same question is also important in the topological-string literature. Already in the simplest case of the conifold, or $A_1$-singularity, the corresponding expression for $\kappa_f$ provides a universal boundary condition \cite{HKTYci,HKQ}, independent of the global geometry. Combined with the holomorphic anomaly equation from \cite{cdg2}, this is a key input in genus-one mirror symmetry \cite{cdg3}. Our first result provides, in broad generality, an intrinsic formula for $\kappa_f$ in terms of local singularity invariants.

\begin{theorem-intro}\label{thm-ADE-intro} 
Let $f\colon X\to\DBbb$ be a projective degeneration of Calabi--Yau manifolds of dimension $n\geq 3$. Assume that $X$ is smooth.

\begin{enumerate}
    \item There is a natural decomposition as a sum of explicit local and global contributions:
$$\kappa_f = \kappa_f^{\mathrm{loc}} + \kappa_f^{\mathrm{gl}}.$$ 
    \item The contribution $\kappa_f^{\mathrm{loc}}$ is a sum of contributions of the singularities of $X_0$, each given in terms of the Hodge theory of the vanishing cohomology. 
    \item $\kappa_f^{\mathrm{gl}}=0$ if $X_0$ has canonical singularities.
\end{enumerate}

\end{theorem-intro}
We refer to Theorem \ref{thm:kappaBCOVgeneral} below for the precise formulation and the description of the local and global invariants.

\subsubsection{} 
When the singularities are isolated, the local contribution of a singular point to $\kappa_f^{\mathrm{loc}}$ hence depends only on the corresponding germ

\begin{displaymath}
f:(\CBbb^{n+1},0)\longrightarrow(\CBbb,0).
\end{displaymath}
Our formula for $\kappa_f^{\mathrm{loc}}$ is actually defined for any such germ, whether or not it arises from a Calabi--Yau degeneration. It can be expressed in terms of invariants derived from the Varchenko--Steenbrink spectrum of $f$, or equivalently from its spectral numbers; see \cite{Varchenko-asymptotic, Steenbrink-mixedonvanishing}. We now introduce these invariants.

We let \begin{displaymath}
    \lambda_{1},\ldots,\lambda_{\mu}\in (0,n+1)
\end{displaymath}
denote the spectral numbers of $f$, where $\mu$ is the Milnor number, and we let $V_{f}$ be Hertling's variance of the spectrum \cite{Hertling}. We also introduce a new invariant $\beta_{f}$ defined by
\begin{equation}\label{eq:def-beta-intro}
    \beta_{f}
    =
    \frac{1}{\mu}
    \sum_{j=1}^{\mu}
    \widetilde{B}_2(\lambda_j).
\end{equation}
Here $\widetilde{B}_2(x) = B_2(\{x\})$ is the 1-periodized version of the second Bernoulli polynomial $B_2(x)=x^2-x+\frac{1}{6}$. Note that $\beta_f$ only depends on the eigenvalues of the monodromy on the Milnor fiber, namely $e^{-2\pi i\lambda_j}$, and is hence a topological invariant.  %The terminology comes from the identity $B_2(\{x\}) = \pi^{-2}\operatorname{Re} \operatorname{Li}_2(e^{2\pi i x})$ and the fact that the numbers $e^{-2\pi i\lambda}$ are the eigenvalues of the monodromy on the cohomology of the Milnor fiber.

\begin{theorem-intro}\label{thm:intro-isolatedkappa_fformula} Let $f: (\CBbb^{n+1},0)\to (\CBbb, 0)$ define the germ of an isolated hypersurface singularity at the origin. Then the following formula holds:   
\begin{equation}\label{eq:kappafloc-intro}
(-1)^{n+1}\frac{\kappa^{\mathrm{loc}}_{f}}{\mu}
    =
    \frac{n-1}{24}
    -\frac12 V_{f}
    +\frac12\beta_{f}.
\end{equation}
\end{theorem-intro}

The above can be seen as an explicit form of Theorem \ref{thm-ADE-intro} for isolated singularities. We refer to Theorem \ref{thm:kappa-f-int-phi} for a proof.

\subsubsection{} Formula \eqref{eq:kappafloc-intro} provides a concrete way to investigate the sign, and hence possibly non-vanishing, of $\kappa_f^{\mathrm{loc}}$. In Section \ref{sec:obstructions}, we carry out a detailed study in several special geometries and uncover a striking positivity phenomenon. Among our results, we single out the following statement.

\begin{theorem-intro}\label{thm:introthmBP}
Let
$f\colon(\CBbb^{n+1},0)\longrightarrow(\CBbb,0)$
define an isolated hypersurface singularity, with $n\geq 3$. Assume that either
\begin{enumerate}
\item  $f$ has minimal spectral number $\lambda_{\min}>3/2$ and satisfies the Hertling variance conjecture (as holds, for example, when $f$ is quasi-homogeneous), or
\item $f=z_0^{a_0}+z_1^{a_1}+\cdots+z_n^{a_n}$, i.e. is a Brieskorn--Pham singularity, which is moreover terminal.
\end{enumerate}
Then
$$(-1)^{n+1}\kappa^{\mathrm{loc}}_f>0.$$
\end{theorem-intro}
In the first case, the condition $\lambda_{\min}>3/2$ implies terminality. We refer to \cite{Hertling} and also \eqref{eq:Hertlinginequality} for the formulation of the Hertling variance conjecture. For Brieskorn--Pham singularities, extending the result to the whole terminal range requires a delicate arithmetic estimate for $\beta_f$, suggesting that $\beta_f$ is an invariant of independent interest.

\subsubsection{} Since terminal singularities are canonical, $\kappa_f$ is purely local for those. As an immediate consequence of Theorem \ref{thm:introthmBP}, we then obtain the following application to the smooth filling problem.
\begin{corollary-intro}
   Let $f\colon X\to \DBbb$ be a projective degeneration of Calabi--Yau manifolds of dimension $n \geq 3$. Assume that $X$ is smooth and $X_0$ is singular and has only terminal singularities of the same form as in Theorem \ref{thm:introthmBP}. Then $f$ does not admit a birational smooth filling. In particular, if $X_0$ has ADE singularities, then there are no birational smooth fillings.
\end{corollary-intro}
We refer to Theorem \ref{thm:conjectureAcases} and Theorem \ref{thm:localkappaBP} for this result, along with other examples. This corollary generalizes the results of Voisin and of Lunardon and Nicaise for $A_1$-singularities \cite{Voisin:fillings, NicaiseLunardon}. It also shows that Hertling's conjecture has concrete geometric consequences, providing, to the best of our knowledge, the first application of this kind. Explicit computations of $\kappa_f^{\mathrm{loc}}$ for ADE singularities are given in Section~\ref{subsec:ADEsingulcomp}, which also includes remarks on when the monodromy is finite.

\subsubsection{} Theorem \ref{thm:introthmBP}, together with further numerical evidence, suggests that the positivity of \eqref{eq:kappafloc-intro} should be governed not by the particular form of the singularity, but by terminality. This leads to the following conjecture for degenerations of Calabi--Yau manifolds. A local counterpart of this conjecture, Conjecture \hyperref[conj:localkappa]{$\mathrm{A}^{\prime}$}, is formulated in Section~\ref{conj:localkappa}.

\begin{conjectureintro}
\label{conjectureintro-kappa-positive}
Let $f\colon X\to\DBbb$ be a projective degeneration of Calabi--Yau manifolds of relative dimension $n\geq 3$. Assume that $X$ is smooth and that $X_0$ is singular and has only isolated terminal singularities. Then
$$ (-1)^{n+1}\kappa_f>0.$$ Consequently, the degeneration admits no birational smooth filling.
\end{conjectureintro}

Several features of the conjecture are worth emphasizing. In dimension three, its conclusion is compatible with a theorem of C.-L.~Wang \cite{wang}, which applies under the weaker assumption that the central fiber is Gorenstein with nef canonical bundle, but does not allow a change of birational model. Moreover, in Conjecture~\ref{conjectureintro-kappa-positive} no polarization is fixed. In the polarized setting, non-existence results for smooth fillings can instead be obtained by methods from the minimal model program \cite{Boucksom:minimal-models-degenerations,Zhang-Yuguang-Survey}.

\subsubsection{} In the known explicit constructions of families admitting no smooth filling, we verify that the invariant $\kappa_f$ is nonzero. Hence  they also do not admit a birational smooth filling. In Proposition \ref {prop:cynkstratenekappa} we treat the example of Cynk and van Straten \cite{CynkStraten}, for which $X_0$ has non-isolated singularities, and in Proposition \ref{prop:tau-bcov-BV} we also treat a new example of Borcea--Voisin type. 

\subsubsection{} Even in the polarized case, these positivity phenomena might have implications for the geometry of moduli of polarized Calabi--Yau manifolds. These, together with the holomorphic anomaly equation from \cite{cdg2}, show that the BCOV invariant provides a quasi-psh exhaustion function on some moduli spaces of Calabi--Yau varieties.

\subsection{Global-to-local mechanism via analytic torsion}

The BCOV invariant is built out of the analytic torsion of holomorphic $p$-differentials. The asymptotics of the BCOV invariant can be deduced from the asymptotics of those analytic torsions. This is a topic of independent interest. Several results hold for non-isolated singularities, but for simplicity and the subsequent applications, we restrict ourselves to the isolated case. 

\subsubsection{}\label{subsubsec:intro-torsion-anal-OX} Let $f\colon X\to\DBbb$ be a projective morphism of complex manifolds, with isolated singularities in the special fiber. We assume that $X$ is endowed with a K\"ahler metric. In our previous article \cite{DURF1}, we obtained the asymptotic behavior of the analytic torsion of the structure sheaf, in the following form:

 \begin{equation}\label{introeq:log-behaviortorsionp=0}
  (-1)^{n}\log \tau(X_t, \Ocal_{X_{t}})=\left(  \frac{\mu}{(n+2)!}-\widetilde{p}_g \right) \log |t|^2 + o(\log|t|).
\end{equation}
The leading term contains a secondary-type version of the geometric genus $p_g$ of the singularities. 

In this article, we generalize this result to more general analytic torsions of holomorphic $p$-forms.  To accomplish this, we complete the work by Yoshikawa \cite{yoshikawa2} and Fang, Lu and Yoshikawa \cite{FLY}, and we apply our previous work \cite{cdg2} and the theory of mixed Hodge modules.

\subsubsection{} 
For the statement, we first define: 
\begin{displaymath}
    \Upsilon(m,p)=\#\{\text{ permutations of } \{1, \ldots, m\} \text{  with $\leq$ } p \text{ descents }\}.
\end{displaymath}
These can be expressed as sums of consecutive Eulerian numbers, see \eqref{eq:Yoshikawa-coefficient}. Also recall the definition of the Steenbrink numbers, built from the mixed Hodge structure of the cohomology of the Milnor fiber $M_f$:
\begin{displaymath}
    s_p =\dim \Gr_F^p H^n(M_f).
\end{displaymath} 
Here, the Milnor fiber is understood as the disjoint union of Milnor fibers of all the individual singularities. 

\subsubsection{} We introduce a secondary version of the Steenbrink numbers, defined from the eigenvalues of the semi-simple part of the monodromy: 
\begin{displaymath}
 \widetilde{s}_p= \frac{1}{2\pi i} \tr\left(\log T_{s} \mid \Gr_F^p H^n(M_f)\right).
\end{displaymath}
Here $\log$ denotes the upper branch of the logarithm whose imaginary part lies in $2\pi [0,1)$. It follows from the definition that $\widetilde{p}_g = \widetilde{s}_n$. These secondary invariants $\widetilde{s}_{p}$ can equivalently be written as sums of spectral numbers, appropriately translated to take values between 0 and 1.

\subsubsection{} 
With this in mind, the following result, proven in Theorem \ref{thm:analytictorsionasymptotics},  generalizes \eqref{introeq:log-behaviortorsionp=0}.

\begin{theorem-intro}\label{thm:asymp-torsion-intro}
With the notation and assumptions as above, we have 
\begin{equation}\label{eq:logtermdescriptioninto}
       (-1)^{n-p}\log \tau(X_t, \Omega^{p}_{X_{t}})=\left(\frac{\Upsilon(n+2,p)}{(n+2)!}\mu-\widetilde{s}_{n-p}-\sum_{k=0}^{p-1}s_{n-k}\right) \log|t|^2+o(\log|t|),
\end{equation}  
as $t\to 0$. 
\end{theorem-intro}

By the construction of the BCOV invariant, $\kappa_f^{\mathrm{loc}}$ is essentially a weighted sum of the logarithmic coefficients appearing in \eqref{eq:logtermdescriptioninto} as $p$ varies. As a matter of fact, we show in \textsection \ref{subsec:isolasymp} that $\kappa_f^{\mathrm{loc}}$ is given as
\begin{equation}\label{eq:kappa_f-formulaintro}
    (-1)^{n+1}\kappa_f^{\mathrm{loc}}=\sum_{2p<n}(n-2p)\left(\frac{\Upsilon(n+2,p)}{(n+2)!}\mu-\widetilde{s}_{n-p}-\sum_{k=0}^{p-1}s_{n-k}\right)-\frac{\mu}{12},
\end{equation}
A reorganization of this sum then yields the formula in Theorem \ref{thm:intro-isolatedkappa_fformula}.

\subsection{Higher Durfee--Saito conjectures}

In our study of the singularities of the analytic torsion in \cite{DURF1}, we were led to conjecture the positivity of the logarithmic coefficient in \eqref{introeq:log-behaviortorsionp=0}. In analogy with the classical Durfee conjecture, we called this the secondary Durfee conjecture. It could be more properly referred to as the secondary Durfee--Saito conjecture, after K. Saito's formulation in higher dimensions of the Durfee conjecture, originally restricted to surfaces. In view of \eqref{eq:logtermdescriptioninto}, one may ask whether our conjecture admits an analogous extension. We will refer to this as the secondary higher Durfee--Saito conjectures.

\subsubsection{} For a germ  $f: (\CBbb^{n+1},0) \to (\CBbb, 0)$ of an isolated hypersurface singularity, our conjecture is stated as follows. \begin{conjectureintro-btilde}[Secondary higher Durfee--Saito conjectures]\label{conjecture:B-tilde}
For $n\geq 1$ and 
\begin{displaymath}
    0\leq p<\frac{n}{2},
\end{displaymath}
one has
\begin{displaymath} 
    \widetilde{s}_{n-p}+\sum_{k=0}^{p-1}s_{n-k} < \frac{\Upsilon(n+2,p)}{(n+2)!}\mu.
\end{displaymath}
\end{conjectureintro-btilde}

The conjecture has been verified numerically in many non-trivial cases. In \textsection\ref{subsec:higher-DS-conj-susp} we establish the stability of the conjecture under double square suspensions, and we use this fact to derive some basic cases from our previous work \cite{DURF1}.

\subsubsection{}\label{subsub:doublesquaresuspensionintro} Conjecture \hyperref[conjecture:B-tilde]{$\widetilde{\mathrm{B}}$} also has direct implications for Conjecture \hyperref[conj:localkappa]{$\mathrm{A}^{\prime}$}. If one supposes that $f$ can be written as a double square suspension $g+u^2+v^2$, a special type of terminal singularity, and if $g$ satisfies Conjecture \hyperref[conjecture:B-tilde]{$\widetilde{\mathrm{B}}$}, in Proposition \ref{prop:doubly-suspended-kappa} we show this provides a special case of Conjecture \hyperref[conj:localkappa]{$\mathrm{A}^{\prime}$}, namely:
\begin{displaymath}(-1)^{n+1}\kappa_f^{\mathrm{loc}} > 0.
\end{displaymath}
In support of Conjecture \hyperref[conjecture:B-tilde]{$\widetilde{\mathrm{B}}$}, we prove in Proposition \ref{prop:doubly-suspended-kappa} this same conclusion unconditionally if $g$ is supposed to be quasi-homogeneous. This relies on a measure-theoretic interpretation of $\kappa_{f}^{\mathrm{loc}}$ discussed in \textsection \ref{subsub:kappafmeasureexpression} below. 

\subsubsection{} The parallel between the Durfee and secondary higher Durfee--Saito conjectures raises the question of whether Conjecture \hyperref[conjecture:B-tilde]{$\widetilde{\mathrm{B}}$} likewise admits a non-secondary counterpart. In \cite{DURF1}, we explained how a suspension argument leads from the original Durfee conjecture to its secondary form. Running this argument in reverse suggests a natural non-secondary version of Conjecture \hyperref[conjecture:B-tilde]{$\widetilde{\mathrm{B}}$}. This can be stated as follows, with the classical case corresponding to $p=0$.

\begin{conjectureintro}[Higher Durfee--Saito conjectures]\label{conj:higherdurfee}
 For $n\geq 2$ and 
\begin{displaymath}
    0\leq p<\left\lfloor\frac{n}{2}\right\rfloor,
\end{displaymath}
one has
\begin{equation}\label{eq:higherDurfeeSaitoconjecture}
    \sum_{k=0}^{p}s_{n-k}
    <
    \frac{\Upsilon(n+1,p)}{(n+1)!}\,\mu.
\end{equation}
\end{conjectureintro}

The same type of verifications that can be done for the secondary version have been done for Conjecture \ref{conj:higherdurfee}. We plan to return to this problem in future work.

\subsection{Spectral inequalities and the distribution of the spectrum}\label{subsec:spectral-measures-intro}

The preceding discussion warrants a more detailed study of the distribution of the spectrum itself. In fact, K. Saito already introduced the uniform discrete probability measure over the spectral numbers, and proposed a comparison program with a continuous model, given by the Irwin--Hall measure, see in particular \cite[(2.5)--(2.8)]{KSaito:distribution}. We can phrase our conjectures in this measure-theoretic framework, developing K. Saito's program and uncovering additional phenomena. This also clarifies a more fundamental measure-theoretic nature of our  conjectures.

\subsubsection{} Following \cite{KSaito:distribution}, we define the spectral measure of an isolated hypersurface singularity $f$ as
\begin{displaymath}
    \nu_f= \frac{1}{\mu}\sum_{i=1 }^\mu \delta_{\lambda_i},
\end{displaymath}
where the sum runs over the spectral numbers of $f$. We also recall that the Irwin--Hall measure $\nu_{\IH_{n}}$ is the probability measure $N(t)dt$ defined by
\begin{equation}\label{def:IH-measure-intro}
  N(t)=\frac{d}{dt}\vol\left((x_0, \ldots, x_n) \in [0,1]^{n+1}\mid \sum x_i \leq t\right).
\end{equation}
By definition, this is the convolution of $n+1$ uniform probability measures on $[0,1]$.

\subsubsection{} The relationship with Conjecture \ref{conj:higherdurfee} is immediate but revealing. One has

$$
\frac{1}{\mu}\sum_{k=0}^{p}s_{n-k}
=
\nu_f([0,p+1]),
$$
whereas a classical relation between the Irwin--Hall distribution and Eulerian numbers gives

$$
\nu_{\IH_{n}}([0,p+1])
=
\frac{\Upsilon(n+1,p)}{(n+1)!}.
$$
Consequently, Conjecture \ref{conj:higherdurfee} is equivalent to
\begin{equation}\label{eq:conj-B-measure}
\nu_f([0,p+1])
<
\nu_{\IH_{n}}([0,p+1]),
\qquad
0\leq p<\left\lfloor\frac{n}{2}\right\rfloor.
\end{equation}
In the language of \cite{KSaito:distribution} this means that any such $p+1$ is a so-called dominant value.
\subsubsection{} The secondary higher Durfee--Saito conjectures also fit in the same picture. If

\begin{equation}\label{eq:function-gp-intro}
g_p(t)
=
(p+1-t)_+-(p-t)_+,
\end{equation}
then Conjecture \hyperref[conjecture:B-tilde]{$\widetilde{\mathrm{B}}$} is equivalent to

\begin{equation}\label{eq:conj-B-tilde-measure}
\int_{\mathbb{R}}g_p\,d\nu_f
<
\int_{\mathbb{R}}g_p\,d\nu_{\IH_{n}},
\qquad
0\leq p<\frac{n}{2}.
\end{equation}
This type of dominance was not envisioned in K. Saito's work.

\subsubsection{} These observations together suggest it would be interesting to investigate to what extent the measures themselves satisfy such inequalities. In this work, we provide a partial answer in terms of the convex order relationship. 

Recall that for two measures $\mu$ and $\nu$ on $\RBbb$, we say that $\mu$ is dominated in convex order by $\nu$, or $\mu \preceq_{\mathrm{cx}} \nu$, if for every convex function
$\varphi\colon\mathbb{R}\longrightarrow\mathbb{R},$ one has
$$
\int_{\mathbb{R}}\varphi\,d\mu
\leq
\int_{\mathbb{R}}\varphi\,d\nu,
$$
whenever the integrals are finite. 
\subsubsection{}
Our first result concerning convex orders is the following theorem, proven in Proposition \ref{prop:convex-order-QH}.  
\begin{theorem-intro}\label{thm:convex-dom-intro}
Let $f$ be a quasi-homogeneous isolated hypersurface singularity of dimension $n\geq 0$. Then

$$
\nu_f\preceq_{\mathrm{cx}}\nu_{\IH_{n}}.
$$

\end{theorem-intro}

We note that whenever $p=0$, the function $g_0$ in \eqref{eq:function-gp-intro} is convex on the support of the measures and hence this implies, together with a simple additional argument on the strictness of the inequality, that the secondary Durfee--Saito conjecture is true for these singularities; see Corollary \ref{cor:Durfee-QH-alternative}. This was already obtained in \cite{DURF1} as an application of the work of Stephen T. Yau and L. Zhang \cite{Yau-Zhang}. The measure-theoretic approach reveals additional layers to this type of problem. We note, however, that the functions $g_p$ for the other values of $p$ are not convex, and hence Theorem \ref{thm:convex-dom-intro} does not imply the higher Durfee--Saito conjectures. 

\subsubsection{}
With this in mind, it seems natural to conjecture that the comparison in Theorem \ref{thm:convex-dom-intro} holds in general: 
\begin{conjectureintro}
    Let $f: (\CBbb^{n+1},0) \to (\CBbb,0)$ define an isolated hypersurface singularity, with $n \geq 0$. Then 
$$        \nu_f\preceq_{\mathrm{cx}}\nu_{\IH_{n}}.
$$
\end{conjectureintro}
By the previous discussion, this conjecture would in particular imply the secondary Durfee--Saito conjecture, with a non-strict inequality. With regard to a structural result, we also remark that the conjecture is stable under Thom--Sebastiani sums of singularities. 
\subsubsection{}\label{subsub:kappafmeasureexpression} It is instructive to reformulate the expression in \eqref{eq:kappa_f-formulaintro} in terms of integrals of the functions \eqref{eq:function-gp-intro}. The particular expression simplifies considerably, and one finds a formula
\begin{displaymath}
    (-1)^{n+1}\frac{\kappa_f^{\mathrm{loc}}}{\mu}=\int \phi_n \left(d \nu_{\IH_{n}}-d\nu_f\right)-\frac{1}{12},
\end{displaymath}
for a piecewise linear convex function $\phi_n$. This type of expression, combined with the convex dominance, is what is behind the unconditional positivity results for $\kappa_f^{\mathrm{loc}}$ in \textsection \ref{subsub:doublesquaresuspensionintro}. In this way, the measure-theoretic constraints on the spectrum lead back directly to the smooth-filling problem with which we began.
 
\section{Hodge theory invariants of singularities}\label{section:Hodgeinvariants}

\subsection{Vanishing cycles}

In this subsection, we briefly recall some facts on vanishing cycles and their mixed Hodge structure. For an introduction to these topics, we refer the reader to \cite[Chapter 4]{Dimca:sheaves-in-topology} and \cite[Chapters 13 \& 14]{Peters-Steenbrink}. Our discussion is based on results of Varchenko \cite{Varchenko-asymptotic}, Steenbrink \cite{Steenbrink-limits, Steenbrink-mixedonvanishing}, Scherk--Steenbrink \cite{Scherk-Steenbrink}, M. Saito  \cite{MSaito:MHP, MSaito:MHM} and Navarro-Aznar \cite{Navarro}.

\subsubsection{} Let $f\colon X \to \DBbb$ be a flat projective morphism of complex manifolds of relative dimension $n$ over the unit disc $\DBbb$. We suppose that $f$ is a submersion over the punctured disc $\DBbb^{\times}$. Associated to $f$ and the locally constant sheaf $\underline{\CBbb}$ on $X$, there is a complex of vanishing cycles $\varphi_{f}(\underline{\CBbb})$ on $X_{0}=f^{-1}(0)$. For simplicity, we will just write $\varphi_f$. This is an object of the derived category of bounded constructible complexes of sheaves on $X_{0}$, and it is supported on the critical locus of $f$. The cohomology of the stalk of $\varphi_{f}$ at $x\in X_0$ is isomorphic to the reduced cohomology of the local Milnor fiber at $x$, denoted by $M_{f, x}$. By M. Saito's theory, there is a structure of mixed Hodge complex on $\varphi_{f}$. Below, we use upper indices for Hodge filtrations, related to the lower indexing in M. Saito's theory by $F^{p}=F_{-p}$. 

\subsubsection{}  The (hyper)cohomology of $\varphi_{f}$ is referred to as the vanishing cohomology, and it carries an induced mixed Hodge structure. It fits into a long exact sequence of mixed Hodge structures
\begin{equation}\label{eq:mixedlongexactvanishingcycles}
    \cdots\to H^{k}(X_{0})\to H^{k}(X_{\infty})\to H^{k}(X_{0},\varphi_{f})\to H^{k+1}(X_{0})\to\cdots.
\end{equation}
Here, $H^{k}(X_{0})$ is endowed with Deligne's mixed Hodge structure on the singular variety $(X_{0})_{\mathrm{red}}$ and $H^{k}(X_{\infty})$ is endowed with Schmid's limit mixed Hodge structure. 

\subsubsection{} In the particular case that $X_0$ has isolated singularities, the cohomology of the vanishing cycles is the direct sum of the reduced cohomologies of the local Milnor fibers at the singular points. Since $\varphi_{f}$ is supported on the singular points, up to quasi-isomorphism, the complex $\varphi_{f}$ is determined by the germ of $f$ at those points. By a globalization argument due to Brieskorn \cite[Section 1.1]{Brieskorn}, this allows us to extend the construction of the complex of vanishing cycles and its mixed Hodge structure to germs of holomorphic functions $(\CBbb^{n+1},0)\to (\CBbb,0)$ defining an isolated singularity at the origin. If we denote by $M$ the corresponding Milnor fiber, we hence obtain a mixed Hodge structure on $H^{n}(M)$.

\subsection{Spectral numbers}\label{subsec:spectral-numbers}
 
In the above setting, the semi-simple part of the monodromy $T_{s}$ acts on the exact sequence of mixed Hodge structures \eqref{eq:mixedlongexactvanishingcycles} with finite order. In particular, for any such cohomology group $H$, we can decompose, as mixed Hodge structures, 
\begin{equation}\label{eq:decomposition-cohomology}
    H=H_{=1}\oplus H_{\neq 1}
\end{equation}
according to the action of $T_s$. The action of $T_s$ and the decomposition \eqref{eq:decomposition-cohomology} can actually be lifted to the level of the mixed Hodge complexes underlying \eqref{eq:mixedlongexactvanishingcycles}, and in particular to $\varphi_{f}$. 

\subsubsection{}\label{subsub:spectralnumber} Let now $\exp(-2 \pi i \lambda)$ be an eigenvalue of $T_s$ acting on $\Gr_{F}^p H^{k}(X_{0}, \varphi_{f})$. The rational number $\lambda$ is fixed by the normalization 
\begin{displaymath}
    k-p<\lambda\leq k-p+1,
\end{displaymath}
and we refer to the set of all such $\lambda$, repeated according to multiplicity, as the spectral numbers of the degeneration on $H^k$. Setting $k=p+q$, we have

    \begin{equation}\label{eq:dim-Hpq-spectral}
       h^{p,q}_{F}:=\dim\Gr_F^{p} H^{k}(X_{0}, \varphi_{f})=\#\{\lambda \hbox{ spectral number on } H^{k},\ k-p < \lambda \leq k-p+1 \}.
    \end{equation}

Following \cite{cdg2}, we introduce the following spectral refinements of the dimensions of the graded Hodge pieces in \eqref{eq:dim-Hpq-spectral}:
\begin{equation}\label{eq:def-alphapq-lower}
    \begin{split}
        \alpha^{p,q}=&-\frac{1}{2\pi i}\tr({^\ell}\log T_s \mid \Gr_F^{p} H^{k}(X_{\infty}))\\
                    &=-\frac{1}{2\pi i}\tr({^\ell}\log T_s \mid \Gr_F^{p} H^{k}(X_{0}, \varphi_{f})).
    \end{split}
\end{equation}
Here, ${^\ell}\log$ is the lower branch of the logarithm, having imaginary part in $2\pi (-1,0]$. The second equality follows from the $T_s$-equivariance of the exact sequence \eqref{eq:mixedlongexactvanishingcycles}, and the fact that the monodromy is trivial on the cohomology of the special fiber, hence
\begin{equation}\label{eq:comparevanishinggeneral}
    \Gr_F^{p} H^k(X_\infty)_{\neq 1} \xrightarrow{\sim} \Gr_F^{p} H^{k}(X_{0}, \varphi_{f})_{\neq 1}.
\end{equation}
In terms of the spectral numbers, we can write
\begin{equation}\label{eq:alphapq}
    \alpha^{p,q} = \sum_{\substack{\lambda \text{ in } \eqref{eq:dim-Hpq-spectral}\\ \lambda\neq k-p+1}} (\lambda - k+p)\in\QBbb\cap [0,h^{p,q}_{F}].
\end{equation}
We also introduce the following alternating sum over the cohomological degrees, with $p$ fixed:
\begin{equation}\label{eq:alphap}
    \alpha^{p} = \sum_{q=0}^{n} (-1)^q \alpha^{p,q}.
\end{equation}

\subsubsection{} The choice of the lower branch of the logarithm in the definition of \eqref{eq:def-alphapq-lower} from \cite{cdg2}, rather than the upper branch, was motivated by an interpretation in terms of elementary exponents of Deligne extensions of Hodge bundles, which control the behavior of the latter under semi-stable reduction, cf. \cite[Theorem D]{cdg2}. Below, it will be convenient to shift to the upper branch.

\subsubsection{}\label{subsubsec:invariants-milnor}  Suppose now that $f: (\CBbb^{n+1},0) \to (\CBbb, 0)$ is a germ of a holomorphic function, defining an isolated singularity at the origin. Denote the associated Milnor fiber by $M_{f}$. The preceding definitions carry over to this setting and are expressed in terms of the cohomology of the Milnor fiber. Since we use the reduced Milnor cohomology throughout, we suppress the tilde in the notation $H^n(M_f)$. Its only nonzero degree is the middle degree. This convention also applies when $n=0$. We refer to the corresponding spectral numbers as the spectrum of $f$.

Following Steenbrink \cite[Section 4]{Steenbrink:Du-Bois}, we define
\begin{equation}\label{eq:def-sk}
    s_{p}=h_{F}^{p,n-p}=\dim\mathrm{Gr}_{F}^{p}H^{n}(M_{f}).
\end{equation}
Hence, in terms of spectral numbers,
\begin{equation}\label{eq:def-sk-bis}
    s_{p}=\#\{\lambda \hbox{ spectral number on } H^{n}(M_{f}),\ n-p < \lambda \leq n-p+1 \}.
\end{equation}
In particular, for the Milnor number, we have
\begin{displaymath}
    \mu=\dim H^{n}(M_{f})=\sum_{p=0}^{n}s_{p}.
\end{displaymath}

We also recall from the introduction the definition of the refined spectral counterpart defined by 
\begin{equation}\label{eq:def-sk-tilde}
    \widetilde{s}_{p}=\frac{1}{2\pi i}\tr(\log T_s \mid \mathrm{Gr}_{F}^{p}H^{n}(M_{f}))\in \QBbb\cap [0,s_{p}],
\end{equation}
where $\log$ is the upper branch of the logarithm, having imaginary part in $2\pi [0, 1)$. Since the map $\lambda \mapsto \lambda'=n+1-\lambda$ interchanges the non-integral spectral numbers determining  $h^{p,q}_F$ and $h_{F}^{q,p}$ as in \eqref{eq:dim-Hpq-spectral} (cf. \cite[\textsection 12.1.3]{Peters-Steenbrink}), we see that 
\begin{equation}\label{eq:sptilde-alphapq}
    \widetilde{s}_{p}=\alpha^{n-p,p}=\sum_{p < \lambda < p+1} (\lambda-p) = \sum_{n-p < \lambda' < n-p+1} (n-p+1-\lambda').
\end{equation}

In the language of \cite{DURF1}, we have that the geometric genus and spectral genus are given by 
\begin{displaymath}
    p_g= s_{n}=h^{n,0}_F
\end{displaymath}
and 
\begin{displaymath}
    \widetilde{p}_g=\widetilde{s}_n=\alpha^{0,n}.
\end{displaymath}

\subsubsection{}\label{subsubsec:invariants-disjoint-union} Occasionally, we may need to emphasize the dependence of the invariants defined above on the function. In this case we write $\mu(f)$, $s_k(f)$, $\widetilde{s}_{k}(f)$, $\widetilde{p}_{g}(f)$. Also, we may extend the definition to disjoint unions of germs of isolated singularities, by simply adding up all the contributions. 

\subsection{Some inequalities of Hodge numbers}

It was proven in \cite[Section 4]{Steenbrink:Du-Bois} that there is an inequality of the form 
\begin{equation}\label{eq:steenbrinkineq}\sum_{k=0}^p s_{k} \leq \sum_{k=0}^{p} s_{n-k}.
\end{equation}
This was revisited in \cite[Theorem 1.11]{Friedman-Laza-isolated}, where it was completed into yet another inequality
\begin{equation}\label{eq:friedmanlazaineq}
    \sum_{k=0}^{p-1} s_{n-k} \leq \sum_{k=0}^{p} s_k.
\end{equation}
This was proven for complete intersections. For the purposes of this article, we need to extend \eqref{eq:steenbrinkineq} and \eqref{eq:friedmanlazaineq}, in the hypersurface case, to an exact computation of the difference.
\begin{proposition} \label{prop:steenbrinkstyleineq}
    Let $f: (\CBbb^{n+1}, 0) \to (\CBbb, 0)$ be the germ of an isolated hypersurface singularity. 
    \begin{displaymath}
        \sum_{k=0}^{p} s_{n-k}- \sum_{k=0}^{p} s_k =  \dim \Gr_F^{p+1} H^n(M_{f})_{=1}=\dim \Gr_F^{n-p} H^n(M_{f})_{=1}
    \end{displaymath}
    and 
    \begin{displaymath}
        \sum_{k=0}^{p} s_{k}- \sum_{k=0}^{p-1} s_{n-k} = \dim \Gr_F^{p} H^n(M_{f})_{\neq 1}= \dim \Gr_F^{n-p} H^n(M_{f})_{\neq 1}.
    \end{displaymath}
\end{proposition}
\begin{proof}
    This follows directly from the symmetry $\lambda \mapsto n+1-\lambda$ of the spectrum and the relations between Hodge numbers and the spectral numbers \eqref{eq:dim-Hpq-spectral}.
\end{proof}
\subsection{Thom--Sebastiani property and suspensions}\label{subsec:TS-properties}
\subsubsection{} Let $f\colon (\CBbb^{n+1},0)\to (\CBbb,0)$ and $g\colon (\CBbb^{m+1},0)\to (\CBbb,0)$ be germs of holomorphic functions defining isolated singularities at the origin. We recall that the Thom--Sebastiani sum of $f$ and $g$ is defined as
\begin{displaymath}
    (f+ g)(x,y)=f(x)+g(y),\quad (x,y)\in \CBbb^{n+1}\times \CBbb^{m+1}.
\end{displaymath}
Let $M_{f}$, $M_{g}$ and $M_{f+g}$ be the Milnor fibers associated to these functions. The theorem of Thom--Sebastiani \cite{Thom-Sebastiani} provides an isomorphism
\begin{equation}\label{eq:TS}
    H^{n+m+1}(M_{f+g})\simeq H^{n}(M_f)\otimes H^{m}(M_g),
\end{equation}
which is compatible with the monodromy action. Varchenko refined this isomorphism at the level of mixed Hodge structures, with the following implication for the spectrum, cf. \cite[Theorem 7.1]{Varchenko-asymptotic}: if we denote the spectral numbers of $f$ and $g$ by $\alpha_i$ and $\beta_j$, respectively, then the spectral numbers of $f+ g$ are given by all the possible sums $\alpha_i+\beta_j$. This is referred to as the Thom--Sebastiani property of the spectrum.

\subsubsection{}\label{subsubsec:thomsebastianisusp} The following statement generalizes \cite[Lemma 7.1]{DURF1}, which dealt with the case $p=0$. 
\begin{lemma}\label{lemma:ThomSebastianisusp}
    Let $f\colon (\CBbb^{n+1},0)\to (\CBbb,0)$ define a germ of an isolated hypersurface singularity at the origin, and let $M$ be the Milnor fiber. Denote by:
    \begin{enumerate}
    \item $N\geq 1$ an integer such that the action of $T_s^N$ is trivial on $H^{n}(M_{f})$.
    \item  $M_{N+1}$ the Milnor fiber of the suspension singularity $f+u^{N+1}=0$ in $(\CBbb^{n+2},0)$. Here, $u$ is an auxiliary variable.
    \end{enumerate} Then, for $0\leq p\leq n$, we have: 
    \begin{displaymath}
        N \left(\widetilde{s}_{n-p}(f)+\sum_{k=0}^{p-1}s_{n-k}(f)\right)=\dim F^{n-p+1}H^{n+1}(M_{N+1}).
    \end{displaymath}
    Moreover, $H^{n+1}(M_{N+1})_{=1}=0.$
\end{lemma}
\begin{proof}
The terms $s_p(f), \widetilde{s}_p(f)$ from \eqref{eq:def-sk}, \eqref{eq:def-sk-tilde} are determined by the spectrum of $f$,  
and by the Thom--Sebastiani property we know that the spectrum of $H^{n+1}(M_{N+1})$ is given in terms of the spectral numbers $\lambda$ of $f$ by 
\begin{equation}\label{eq:newspectrum} 
    \lambda+\frac{k}{N+1}, 1 \leq k \leq N.
\end{equation} 
If $\lambda \leq p$ then for all $1 \leq k \leq N$ the expression \eqref{eq:newspectrum} is strictly less than $p+1$ and hence each such $\lambda$ contributes $N$ times to $\dim F^{n-p+1}H^{n+1}(M_{N+1})$. Now consider the case $p<\lambda\leq p+1$. Multiplying \eqref{eq:newspectrum} by $N$, we find
\begin{displaymath}
    N\lambda+\frac{N k}{N+1}.
\end{displaymath}
This quantity is $\leq N(p+1)$ exactly for $k=1,\ldots, N p + N - N \lambda $. Hence each such $\lambda$ contributes $N(p+1-\lambda)$ times to $\dim F^{n-p+1}H^{n+1}(M_{N+1})$. On the other hand, by \eqref{eq:sptilde-alphapq},
\begin{displaymath}
    \widetilde{s}_{n-p}(f) =\sum_{p<\lambda\leq p+1} (p+1-\lambda),
\end{displaymath}
from which one concludes the first part. The statement about the 1-eigenvalue is immediate, since $\lambda+k/(N+1)$ can't be an integer by the choice of $N$. \end{proof}

\subsection{Examples of singularities with finite monodromy}\label{subsec:finitemonodromy}
Let $f\colon X\to\DBbb$ be a projective morphism of complex manifolds, of relative dimension $n\geq 1$, which is a submersion over $\DBbb^\times$. We suppose that $X_0$ has isolated singularities. In this case, the monodromy on $H^k(X_\infty)$ is trivial for $k\neq n$. We discuss situations for which the monodromy on $H^n(X_\infty)$ is also finite.  

\subsubsection{} We first record a simple criterion to deduce the finiteness of the monodromy on $H^n(X_\infty)$ from the cohomology of the Milnor fiber $M$.
\begin{lemma}\label{lemma:finitemonodromy}
Suppose the monodromy acting on $H^n(M)$ is finite and that $H^n(M)_{=1}=0$. Then the monodromy on $H^n(X_\infty)$ is finite.
\end{lemma}

\begin{proof}
This follows from combining \eqref{eq:decomposition-cohomology} and \eqref{eq:comparevanishinggeneral}.
\end{proof}

\subsubsection{}\label{subsubsec:monodromy-quasi-homogeneous} A natural class of examples where the monodromy on the Milnor fiber can be shown directly to be finite is that of quasi-homogeneous singularities, see the introduction of \cite{DimcaMilnorweightedhomo}. For the formulation, let $(w_0,\dots,w_n)$ be rational weights in $(0,1)$. Choose a common denominator $N$ so that $Nw_i$ is integral for every $i$, and define
\[
\lambda \cdot x := (\lambda^{Nw_0}x_0,\dots,\lambda^{Nw_n}x_n).
\]
We say that a germ $g\colon (\CBbb^{n+1},0)\to (\CBbb,0)$ is quasi-homogeneous if
\[
    g(\lambda \cdot x)=\lambda^N g(x).
\]
In this case, the Milnor fiber of $g$ has finite monodromy of order dividing $N$. The spectral numbers $\lambda_{i}$ are determined from the following identity for the spectral polynomial with rational exponents:
\begin{equation}\label{eq:qhomspectrum}
    \operatorname{Sp}_{g}(T)=\sum_{i=1}^{\mu} T^{\lambda_{i}}=\prod_{i=0}^{n} \frac{T^{w_i}-T}{1-T^{w_i}}.
\end{equation}
Here, $\mu$ is the Milnor number of $g$, given by
\begin{displaymath}
    \mu=\prod_{i}\left(\frac{1}{w_{i}}-1\right).
\end{displaymath}

If $f\colon X\to\DBbb$ is a projective morphism as above, we say that it has quasi-homogeneous singularities if, locally around the singular points in $X_{0}$, the morphism is quasi-homogeneous in some system of local coordinates. Then, to conclude that the monodromy on $H^{n}(X_\infty)$ is finite, by Lemma \ref{lemma:finitemonodromy} it suffices to verify that no exponent occurring in \eqref{eq:qhomspectrum} is an integer.

\subsubsection{} Brieskorn--Pham singularities are quasi-homogeneous. The expression in \eqref{eq:qhomspectrum} shows that the spectral numbers are given by 
\begin{equation}\label{eq:BP-integral-point-countes}
     \frac{k_0}{a_0} + \frac{k_1}{a_1} + \ldots + \frac{k_n}{a_n},\quad 0 < k_i < a_i.
\end{equation}
It is elementary to verify that if the integers $a_0, \ldots, a_n$ are pairwise coprime this sum is never an integer. For otherwise,  multiplying \eqref{eq:BP-integral-point-countes} by $\prod_{i}a_{i}$ and reducing modulo any $a_j$ one finds a contradiction. Hence in this case, $H^n(M)_{=1}=0$.

\subsubsection{} To give more concrete examples, we give the following version of a result of L\^e D\~{u}ng Tr\'ang \cite{TranNoeuds}:
\begin{lemma}\label{lemma:ledungtrang}
    Let $g\colon (\CBbb^2,0)\to (\CBbb,0)$ be a germ of a holomorphic function, such that $g=0$ defines a germ of an irreducible isolated curve singularity at the origin. Then, for every quadratic form $Q(u_1,\ldots, u_{2n})=\sum_{i=1}^{2n}u_{i}^{2}$, the monodromy of the Milnor fiber of the Thom--Sebastiani sum $g+ Q$ is finite, and $1$ is not an eigenvalue.
    
\end{lemma}
\begin{proof}
The case $n=0$ is a combination of \cite[proof of Th\'eor\`eme 3.2.4 \& Th\'eor\`eme 3.3.1]{TranNoeuds}. The first one entails that 1 is not an eigenvalue and the second one says that the monodromy is finite. In the general case, by the Thom--Sebastiani theorem, there is an isomorphism 
\[
     H^1(M_g)\otimes H^{2n-1}(M_{Q})  \simeq H^{2n+1}(M_{g+ Q}),
\]
compatible with the action of the monodromy. By the discussion in \textsection\ref{subsubsec:monodromy-quasi-homogeneous} on the monodromy of the quasi-homogeneous singularities, and in particular \eqref{eq:qhomspectrum}, the monodromy on 
\begin{equation}\label{eq:monodromysuspsension} 
    H^{2n-1}(M_{Q})
\end{equation}
is trivial. Hence the monodromy of $g + Q$ satisfies the same stated properties as $g$.
\end{proof}

In the global setting of a projective degeneration $f\colon X\to\DBbb$ as above, if locally around the singular points $f$ has the form of Lemma \ref{lemma:ledungtrang}, in some system of local coordinates, then the monodromy on $H^{n}(X_{\infty})$ is finite.

\section{Comparisons of Kähler differentials}\label{sec:comparison-Kahler}

Let $f\colon X\to S$ be a flat, projective morphism of complex manifolds, of relative dimension $n$. We suppose that $S$ is one-dimensional and connected, and that $f$ is a submersion over $S^{\circ}=S\setminus\lbrace 0\rbrace$, for some $0\in S$. We write $X^{\circ}=f^{-1}(S^{\circ})$. For every integer $p\geq 0$, the sheaf $\Omega^{p}_{X^{\circ}/S^{\circ}}$ is locally free and we can consider its Knudsen--Mumford determinant of cohomology $\lambda(\Omega^{p}_{X^{\circ}/S^{\circ}})$. In this section, we compare natural extensions of this line bundle to $S$, called K\"ahler and logarithmic extensions. 

 In order to be consistent with some references, we initially assume that $f$ arises as the restriction of a morphism of algebraic varieties to an open analytic subset of the base space. In the particular case when $X_0$ has isolated singularities, we explain in \textsection\ref{subsec:isolated-singularities} how to bypass the algebraicity restriction.  

For background on the determinant of the cohomology, the reader is referred to Knudsen--Mumford \cite{KnudsenMumford}, formulated in the algebraic setting, and Bismut--Gillet--Soul\'e \cite{BGS3} and Bismut--Bost in the complex analytic setting \cite[Section 4]{bismutbost}. Both constructions compare via the analytification functor.

To lighten notation, we use additive notation for the tensor product of line bundles on the base, i.e. we may write $L+M$ instead of $L\otimes M$ for line bundles $L$ and $M$.

\subsection{The Kähler extension}\label{subsec:Kahler-extension}

Following Fang--Lu--Yoshikawa \cite[Section 5.1]{FLY}, the natural morphism $f^{\ast}\Omega_{S}\to\Omega_{X}$ induces the differential maps of a complex
\begin{equation}\label{eq:Kahler-complex}
    \widetilde{\Omega}^{p}_{X/S}\colon (f^{\ast}\Omega_{S})^{\otimes p}\to(f^{\ast}\Omega_{S})^{\otimes (p-1)}\otimes\Omega_{X}\to\cdots\to
    f^{\ast}\Omega_{S}\otimes\Omega_{X}^{p-1}\to\Omega_{X}^{p}.
\end{equation}
There is a canonical morphism $\widetilde{\Omega}^{p}_{X/S}\to\Omega_{X/S}^{p}$, which is a quasi-isomorphism over $S^{\circ}$. Hence, the determinant bundle
\begin{equation}\label{eq:determinant-Kahler-complex}
    \begin{split}
    \lambda(\widetilde{\Omega}^{p}_{X/S})&=\sum_{j=0}^{p}(-1)^{j}\lambda(f^{\ast}\Omega_{S}^{\otimes j}\otimes\Omega_{X}^{p-j})\\
    &\simeq\sum_{j=0}^{p}(-1)^{j}j\chi(\Omega^{p-j}_{X|X_{\infty}})\ \Omega_{S} + \sum_{j=0}^{p}(-1)^{j}\lambda(\Omega_{X}^{p-j})
    \end{split}
\end{equation}
is a line bundle extension of $\lambda(\Omega_{X^{\circ}/S^{\circ}})$ to $S$: there is a canonical isomorphism
\begin{equation}\label{eq:kahler-extension}
    \lambda(\widetilde{\Omega}^{p}_{X/S})_{\mid S^{\circ}}\simeq\lambda(\Omega_{X^{\circ}/S^{\circ}}^{p}).
\end{equation}
We refer to this extension as the \emph{K\"ahler extension} of $\lambda(\Omega_{X^{\circ}/S^{\circ}}^{p})$. By \cite[Corollary 1.2.7]{Kato-Saito} it can equivalently be interpreted as the determinant of the cohomology of the left derived exterior power $L\wedge^{p}\Omega_{X/S}$, that is
\begin{displaymath}
    \lambda(\widetilde{\Omega}^{p}_{X/S})\simeq\lambda(L\wedge^{p}\Omega_{X/S}).
\end{displaymath}

\subsection{The logarithmic extensions}
We next introduce extensions of the determinant bundles of sheaves of differentials, obtained from normal crossings models and sheaves of logarithmic forms.
\subsubsection{}  

Let $g\colon Y\to S$ be a normal crossings model of $f\colon X\to S$, obtained by performing an embedded resolution of singularities of $X_0$ in $X$. We consider the sheaf of relative logarithmic differentials $\Omega_{Y/S}(\log)=\Omega_{Y}(\log Y_{0})/g^{\ast}\Omega_{S}(\log [0])$, and write $\Omega^{p}_{Y/S}(\log)$ for its $p$-th exterior power. We recall that the relative cohomology sheaves of $\Omega^{p}_{Y/S}(\log)$ are locally free, and are called the Hodge bundles. Up to unique isomorphisms, these don't depend on the choice of normal crossings model. We refer to the work of Steenbrink for these facts \cite{Steenbrink-limits, Steenbrink-mixedonvanishing}, and to \cite[Section 2]{cdg2} for a summary. Consequently, up to unique isomorphism, the determinant of the cohomology $\lambda(\Omega^{p}_{Y/S}(\log))$ is independent of the normal crossings model too. This is \emph{the logarithmic extension} of $\lambda(\Omega_{X^{\circ}/S^{\circ}}^{p})$. 

Similar to \eqref{eq:Kahler-complex}, we have a  complex of vector bundles which is naturally quasi-isomorphic to $\Omega^{p}_{Y/S}(\log)$:
\begin{equation}\label{eq:log-complex}
     (g^{\ast}\Omega_{S}(\log))^{\otimes p}\to(g^{\ast}\Omega_{S}(\log))^{\otimes (p-1)}\otimes\Omega_{Y}(\log)\to\cdots\to g^{\ast}\Omega_{S}(\log [0])\otimes\Omega_{Y}^{p-1}(\log)\to\Omega_{Y}^{p}(\log).
\end{equation}
In particular, we find that 
\begin{equation}\label{eq:determinant-log-complex}
    \begin{split}
    \lambda(\Omega_{Y/S}^p(\log)) &\simeq \sum_{j=0}^p \lambda(g^\ast \Omega_S^j(\log) \otimes \Omega_{Y}^{p-j}(\log))^{(-1)^j}\\
    &\simeq {\sum_{j=0}^p j (-1)^j \chi(\Omega_{X|_{X_\infty}}^{p-j})}\ \Omega_S(\log)+ \sum_{j=0}^p (-1)^j \lambda(\Omega_{Y}^{p-j}(\log)).
    \end{split}
\end{equation}

\subsubsection{}  
The above is sometimes called the lower logarithmic extension since it corresponds to the lower branch of the logarithm in the context of Deligne extensions of local systems. We will also actually consider the upper logarithmic extension. The lower and upper extensions correspond to each other through Grothendieck--Serre duality, cf. \cite[Section 2]{FreeMori}  or \cite[Proposition 2.9]{Kollar2}:

\begin{displaymath}
    {^{u}}\lambda(\Omega_{X^{\circ}/S^{\circ}}^p) \simeq \lambda(\Omega_{Y/S}^{n-p}(\log))^{(-1)^{n+1}}
    \simeq\lambda(\Omega_{Y/S}^{n-p}(\log)^\vee \otimes\omega_{Y/S}),
\end{displaymath}
where the decoration $u$ indicates the upper extension. Since $\Omega_{Y/S}^{n-p}(\log)^\vee \simeq \Omega^p_{Y/S}(\log) \otimes \omega_{Y/S}(\log)^\vee$ and $\omega_{Y/S}(\log)=\omega_{Y/S}(Y_{0, \mathrm{red}}-Y_{0})$, we in fact have a natural isomorphism 
\begin{equation} \label{eq:upperextensioniso}{^{u}}\lambda(\Omega_{X^{\circ}/S^{\circ}}^p) \simeq \lambda(\Omega^p_{Y/S}(\log) \otimes \Ocal(Y_0-Y_{0,\mathrm{red}})).
\end{equation}

\subsection{Comparisons of extensions}\label{subsec:comparisons}
In this section, for sheaves of holomorphic differentials, we wish to compare the K\"ahler and logarithmic extensions of the determinants of the cohomology. We will use the constructions in Section \ref{section:Hodgeinvariants}, and in particular the complex of vanishing cycles. For this, restricting $S$ if necessary, we choose a holomorphic coordinate $t$ centered at $0$. The numerical invariants introduced in \emph{loc. cit.} do not depend on the choice of coordinate. 

\subsubsection{}  
Following \cite[Section 3.3]{cdg2}, we define an integer $\mu_p$ by the canonical isomorphism comparing the K\"ahler and the logarithmic extensions:
\begin{equation}\label{eq:def-mup}
    \lambda(\widetilde{\Omega}^{p}_{X/S})=\lambda(\Omega^{p}_{Y/S}(\log))+\mu_{p}.
\end{equation}
Here and below, to lighten notation, we write $L+k$ instead of $L+\Ocal(k[0])$ for a line bundle $L$ on $S$ and an integer $k$. Hence if $M = L+k$ and $\sigma$ is a trivializing section of $L$, then $\sigma \cdot t^{-k}$ is a trivializing section of $M$.

It is our purpose to provide a general expression for $\mu_{p}$ in terms of mixed Hodge structures.

\begin{theorem}\label{thm:mu-p-general}
Let $f\colon X\to S$ be as above. Then 

\begin{displaymath}
   \begin{split}
    \mu_p &=(-1)^{p-1} \chi(F^{n-p+1}\varphi_f)\\
    &=(-1)^{p-1} \sum_j (-1)^{j}\dim F^{n-p+1} H^j(X_0, \varphi_f).
    \end{split}
\end{displaymath}
\end{theorem}

\begin{proof}    
    Let $\pi\colon Y\to X$ be the log-resolution map of $i:X_{0, \mathrm{red}} \to X$, and set $E = (\pi^{-1}(X_0))_{\mathrm{red}}$. Then $ \Omega_{Y}(\log )=\Omega_{Y}(\log E)$. Comparing \eqref{eq:determinant-Kahler-complex} and \eqref{eq:determinant-log-complex} we find that, similarly to \cite[Proposition 3.7, especially (3.6) in the proof]{cdg2}:    \begin{equation}\label{eq:compare-log-and-kahler}
    -\mu_p=\lambda(\Omega_{Y/S}^p(\log))-\lambda(\widetilde{\Omega}_{X/S}^p) = \sum_{k=0}^{p-1} (-1)^{p-k} \chi(\Omega_{X_\infty}^{k}) +  \sum_{k=0}^p (-1)^{p-k}\left[\lambda(\Omega_{Y}^{k}(\log E)) - \lambda(\Omega_X^{k})\right] .
    \end{equation}
To compute the latter difference, we note that according to \cite[Proposition 3.3]{SteenbrinkVanishingThm}\footnote{This is where we use the algebraicity assumption on $f$.}, there is a distinguished triangle: 
\begin{equation}\label{eq:distinguishedtriangleOmegalog}
 R\pi_\ast (\Omega^{n+1-k}_{Y}(\log E)(-E)) \to \Omega_X^{n+1-k} \to i_{\ast}\underline{\Omega}_{X_{0, \mathrm{red}}}^{n+1-k} ,
    \end{equation}
    where $\underline{\Omega}^m_{X_{0, \mathrm{red}}}$ denotes the $m$-th graded quotient of the Du Bois-complex of $X_{0, \mathrm{red}}$, shifted by $m$, cf. \cite{duBois}. For a variety $Z$ admitting a dualizing sheaf $\omega_{Z}$ and a complex $\mathcal{F}$ of coherent sheaves on $Z$, we denote by $\DBbb(\mathcal{F})=R\mathcal{H}om(\mathcal{F}, \omega_Z)$. We will use this for the non-singular varieties $X$, $Y$, and for $X_{0,\mathrm{red}}$, which is a divisor in $X$, hence admits a dualizing sheaf by adjunction. With this in mind, by Grothendieck duality, we have
    \begin{displaymath}
        \mathbb{D}(R\pi_\ast (\Omega^{n+1-k}_{Y}(\log E)(-E))\simeq R\pi_\ast (\Omega^{k}_{Y}(\log E)).
    \end{displaymath} 
    Hence, by applying $\DBbb$ to \eqref{eq:distinguishedtriangleOmegalog} we conclude a similar distinguished triangle: 
    \[
        i_{\ast}\mathbb{D}(\underline{\Omega}_{X_0, \mathrm{red}}^{n+1-k})[-1] \to \Omega_X^{k} \to R\pi_\ast (\Omega^{k}_{Y}(\log E)) ,
    \]
where we have used the relationship $\mathbb{D} (i_{\ast}\underline{\Omega}_{X_0, \mathrm{red}}^{n+1-k})=i_{\ast}\mathbb{D}(\underline{\Omega}_{X_0, \mathrm{red}}^{n+1-k})[-1]$. Hence we find that 
    \[ 
        \lambda(\Omega^{k}_{Y}(\log E))-\lambda(\Omega^{k}_{X})=\chi(\mathbb{D}(\underline{\Omega}^{n+1-k}_{X_0,\mathrm{red}})) .
    \]
Again by Grothendieck duality we have 
$$
    \chi(\Omega_{X_{\infty}}^{k})=(-1)^{n}\chi(\Omega_{X_{\infty}}^{n-k}), \quad  \chi(\mathbb{D}(\underline{\Omega}^{n+1-k}_{X_0,\mathrm{red}})) = (-1)^{n} \chi(\underline{\Omega}^{n+1-k}_{X_0,\mathrm{red}}).
$$
Inserting these relationships into the right hand side of \eqref{eq:compare-log-and-kahler}, and reindexing the second sum, we see that the difference is given by:
    \begin{equation}\label{eq:goingforvanishing}
        \sum_{k=0}^{p-1} (-1)^{n-p+k} \chi(\Omega_{X_\infty}^{n-k}) - \sum_{k=0}^{p-1} (-1)^{n-p+k} \chi(\underline{\Omega}^{n-k}_{X_0,\mathrm{red}})
    \end{equation}
Since the filtered Du Bois complex computes the singular cohomology with Deligne's Hodge filtration \cite[Th\'eor\`eme 4.5]{duBois}, we have that 
\[
\chi(\underline{\Omega}^{n-k}_{X_0,\mathrm{red}})= \sum_{j} (-1)^j \dim H^j(X_0,\underline{\Omega}^{n-k}_{X_0,\mathrm{red}} )=(-1)^{n-k} \sum_{j} (-1)^j \dim \Gr_{F}^{n-k} H^j(X_0)
\]
and analogously for $\Omega_{X_\infty}^{n-k}$. Inserting these rewritings into  \eqref{eq:goingforvanishing}, and utilizing that the long exact sequence \eqref{eq:mixedlongexactvanishingcycles} is a long exact sequence of mixed Hodge structures, we finally find that:
\[
-\mu_p=(-1)^p \sum_j (-1)^j \dim F^{n-p+1} H^j(X_0, \varphi_f ).
\]
Keeping track of indices, one sees that this can be recast as
$$
    (-1)^{p}\chi(F^{n-p+1}\varphi_{f}).
$$

\end{proof}
\subsubsection{}  

For completeness, we also provide the comparison between the K\"ahler and the upper logarithmic extensions. This won't play much of a role in the sequel. 

Define $\mu_p^+$ as the comparison
\begin{displaymath}
\lambda(\widetilde{\Omega}^{p}_{X/S})={^{u}}\lambda(\Omega_{X^{\circ}/S^{\circ}}^p) +\mu_{p}^+.
\end{displaymath}
Then, an argument analogous to that in Theorem \ref{thm:mu-p-general} proves the following statement.
\begin{proposition}\label{prop:mu-p-upper}
    Let $X\to S$ be as above. Then 
    \begin{displaymath}
        \begin{split}
        \mu_p^+=&(-1)^{p-1} \chi(\varphi_f/F^{p+1}) \\
       =& (-1)^{p-1} \sum_{j} (-1)^j \dim (H^{j}(X_0, \varphi_f)/F^{p+1}),
        \end{split}
    \end{displaymath}
    where $F^{p+1}$ is a shorthand for $F^{p+1}H^{j}(X_0, \varphi_f)$.
\end{proposition}

\subsubsection{}  
Consider the decomposition $\varphi_{f}=\varphi_{f,=1}\oplus\varphi_{f,\neq 1}$ according to the action of $T_{s}$, which induces the corresponding decomposition on the vanishing cohomology, cf. \textsection\ref{subsec:spectral-numbers}. The following statement can be interpreted as a Serre duality type statement.
\begin{proposition}\label{prop:mup+-mup}
 We have the relationship
 \begin{displaymath}
    \mu_{p}^{+}-\mu_{p}=(-1)^{p-1}\chi(\Gr^{p}_F\varphi_{f, \neq 1}).
 \end{displaymath}
\end{proposition}
\begin{proof}
    The upper and lower extensions of the Hodge bundles $\Hcal^{p,q}$ on $S^{\circ}$, denoted by ${^u}\Hcal^{p,q}$ and ${^\ell}\Hcal^{p,q}$ respectively, can be compared with the semi-stable reduction case (cf. \cite[Section 2]{cdg2}). In each case, the comparison is controlled by the non-trivial semi-simple monodromy eigenvalues. More precisely, one finds that
    \begin{displaymath}
        \det( {^u}\mathcal{H}^{p,j-p})-\det ({^\ell}\mathcal{H}^{p,j-p})=\dim \Gr_F^p H^{j}(X_\infty)_{\neq 1}\overset{\eqref{eq:comparevanishinggeneral}}{=}\dim \Gr_F^p H^{j}(\varphi_f )_{\neq 1}.
    \end{displaymath}
    This means that,  
    \begin{displaymath}
        \begin{split}{^{u}}\lambda(\Omega_{X^{\circ}/S^{\circ}}^p) - \lambda(\Omega_{Y/S}^p(\log)) &= \sum_{j} (-1)^{j-p} \left(\det( {^u}\mathcal{H}^{p,j-p})-\det ({^\ell}\mathcal{H}^{p,j-p})\right) \\
        & = (-1)^p \sum_{j} (-1)^j \dim \Gr_F^p H^{j}(\varphi_f )_{\neq 1}\\
        & =(-1)^{p}\chi(\Gr^{p}_F\varphi_{f,\neq 1}).
        \end{split}
    \end{displaymath}
\end{proof}

\subsection{Case of isolated singularities}\label{subsec:isolated-singularities}
We specialize, and slightly generalize, the previous section to the context of degenerations with isolated singularities.

\subsubsection{} Suppose now that $f\colon X\to S$ has at most isolated singularities in $X_{0}$. Recall that, for consistency with some references, so far we assumed that $f$ is the restriction of a morphism of algebraic varieties. At least in the isolated singularity setting, this can be easily relaxed.

\begin{lemma}\label{lemma:isolated-removal-algebraicity}
The results in \textsection\ref{subsec:comparisons} hold for projective morphisms of complex manifolds, with isolated singularities, not necessarily arising as the restriction of a morphism over an algebraic curve.
\end{lemma}
\begin{proof}
For the sheaves of holomorphic differentials, the K\"ahler extension and the logarithmic extensions of the determinant of the cohomology are similarly defined. We need to comment on the comparison. For concreteness, we explain how to deal with the analog of Theorem \ref{thm:mu-p-general}. For simplicity, we suppose that $X_{0}$ has a unique singular point $x$. We adopt the same notation as in the proof of that statement. 

Denote the complex \eqref{eq:log-complex} by $\widetilde{\Omega}_{Y/S}^{p}(\log)$. Pulling back differential forms induces a canonical morphism of complexes
\begin{displaymath}
    \widetilde{\Omega}_{X/S}^{p}\to R\pi_{\ast}\widetilde{\Omega}_{Y/S}^{p}(\log),
\end{displaymath}
where $\pi\colon Y\to X$ is the log-resolution of the special fiber. Let $\Ccal_{f}$ be the cone of this map. This complex is supported on the singular point $x\in X_{0}$, and $\mu_{p}=-\chi(\Ccal_{f})$. Since $\Ccal_{f}$ is supported on $x$, if $U$ is an open neighborhood of $x$ in $X$, then $\mu_{p}=-\chi(\Ccal_{f\mid U})$. 

Since $x$ is an isolated singularity, by Brieskorn's globalization argument \cite[Section 1.1]{Brieskorn} we can find a flat projective morphism of non-singular algebraic varieties $g\colon Z\to T$, such that:
\begin{enumerate}
    \item $T$ is a Zariski open neighborhood of $0$ in $\ABbb^{1}_{\CBbb}$.
    \item $g$ is smooth over $T\setminus\lbrace 0\rbrace$, and $Z_{0}$ has a unique singular point $z$.
    \item There are analytic open neighborhoods $V\subset Z$ of $z$, $U\subset X$ of $x$, and a biholomorphism $\psi\colon V\to U$, such that $f\circ\psi=g$. 
\end{enumerate}
Since the log-resolution $\pi$ is determined by a sequence of blow-ups over $x$, we can find a log-resolution of $Z_{0}$, say $\pi^{\prime}\colon W\to Z$, such that $\pi^{\prime}_{\mid V}$ gets identified with $\pi_{\mid U}$ via $\psi$. Then, if we construct $\Ccal_{g}$ analogously to $\Ccal_{f}$, there is a natural quasi-isomorphism $\psi^{\ast}\Ccal_{f\mid U}\to \Ccal_{g\mid V}$. Consequently, $\mu_{p}=-\chi(\Ccal_{g\mid V})$. Now, the latter equals $(-1)^{p-1} \chi(F^{n-p+1}\varphi_g)$, by Theorem \ref{thm:mu-p-general}. The isomorphism $\psi$ also induces a quasi-isomorphism $\psi^{\ast}\varphi_{f\mid U}\to\varphi_{g\mid V}$, and we conclude that $\mu_{p}=(-1)^{p-1} \chi(F^{n-p+1}\varphi_f)$. 
\end{proof}

\section{Invariants of singularities and asymptotics of analytic torsion}\label{section:torsion-differentials}
In this section we extend the discussion of our previous work \cite{DURF1} to the case of analytic torsion of holomorphic $p$-forms. Concretely, we relate the dominant terms of the asymptotics with the invariants of singularities introduced in Section \ref{section:Hodgeinvariants}. 

\subsection{Holomorphic analytic torsion}
We begin by briefly recalling the definition of holomorphic analytic torsion. We refer to \cite{BGS1} and \cite[Chapter IV, Section 3]{Soule:lectures} for further details.

\subsubsection{} Let $(X,\omega)$ be a compact K\"ahler manifold of dimension $n$, and $(E,h)$ a holomorphic Hermitian vector bundle on $X$. The holomorphic analytic torsion associated to this data is the weighted alternating product
\begin{equation}\label{def:analytictorsion}
    \tau(X,\omega,E,h)=\prod_{q=0}^{n}(\det\Delta^{0,q}_{\ov{\partial}})^{(-1)^{q}q}\ \in\RBbb_{>0},
\end{equation}
where $\Delta^{0,q}_{\ov{\partial}}$ is the Dolbeault Laplacian acting on $A^{0,q}(X,E)$, depending on the chosen Hermitian structures, 
\begin{displaymath} \det\Delta^{0,q}_{\ov{\partial}} = \exp\left(-\zeta_{0,q}'(0)\right)
\end{displaymath}
is its zeta-regularized determinant, where 
\begin{equation}\label{def:zetafunction}
    \zeta_{0,q}(z)=\sum_{\lambda \in \Spec'(\Delta_{\overline{\partial}}^{0,q})} \frac{1}{\lambda^z}
\end{equation}
is the spectral zeta function over non-zero eigenvalues of the $(0,q)$-Laplacian. If the K\"ahler metric $\omega$ and the Hermitian metric $h$ are clear from the context, we will abbreviate the notation by writing $\tau(X,E)$ instead of $\tau(X,\omega,E,h)$.

The Quillen metric on the determinant of the cohomology $\lambda(E)$ of $E$ is defined as
 \begin{equation}\label{eq:quillenmetricdef}
    h_{Q}=\tau(X,E)\cdot h_{L^{2}},
\end{equation}
where $h_{L^{2}}$ is the $L^2$ metric from Hodge theory. The norms associated to $h_{L^{2}}$ and $h_{Q}$ are denoted by $\|\cdot\|_{L^{2}}$ and $\|\cdot\|_{Q}$, respectively.

\subsubsection{} These constructions can be put in families. For this, suppose now that $(X,\omega)$ is a K\"ahler manifold, and $f\colon X\to S$ is a proper submersion of complex manifolds with $n$-dimensional fibers. For a Hermitian vector bundle $(E,h)$ on $X$, the fiberwise Quillen metric, computed with respect to the fiber restrictions of $\omega$ and $h$, defines a smooth Hermitian metric on $\lambda(E)$. In general, neither the $L^{2}$-metric, nor the analytic torsion, depend smoothly on $t\in S$. However, if the $L^{2}$-metric is smooth, then $t\mapsto\tau(X_{t},E|_{X_{t}})$ is a smooth function on $S$. This arises when all the higher direct images $R^{q}f_{\ast}E$ are locally free, as for instance if $E=\Omega_{X/S}^{p}$, by Hodge theory. Hence, if we endow $\Omega_{X/S}^{p}$ with the Hermitian metric induced by $\omega$, we conclude that the function
\begin{equation}\label{eq:torsion-anal-p}
    t\mapsto\tau(X_{t},\Omega_{X_{t}}^{p}),\quad t\in S,
\end{equation}
is $\Ccal^{\infty}$ on $S$. 

\subsection{Singularities of the analytic torsion}
For the sheaves of holomorphic differentials, we describe the asymptotic behavior of the holomorphic analytic torsion, for projective degenerations with isolated singularities. The main tools are Yoshikawa's theorem on the singularities of the Quillen metric, and refinements of Schmid's asymptotics of $L^2$-metrics. 

\subsubsection{}  

The setting is as follows. Let $f\colon X\to \DBbb$ be a flat projective morphism of complex manifolds. Let $n$ be the relative dimension of $f$. We assume that $f$ is a submersion over $\DBbb^{\times}$, and that $X_0$ has isolated singularities. We write $X^\times = f^{-1}(\DBbb^\times )$. Let $\omega$ be a fixed K\"ahler metric on $X$. For every integer $p\geq 0$, we endow the vector bundle $\Omega_{X^{\times}/\DBbb^{\times}}^{p}$ with the Hermitian metric induced by $\omega$.  Associated to these choices, we consider the Quillen metric on $\lambda(\Omega^{p}_{X^{\times}/\DBbb^{\times}})$. We wish to describe the singularity of this metric close to $0$, with respect to the Kähler extension.

\subsubsection{}\label{subsubsec:Kahler-extension} 
We first treat the singularities of the Quillen metric. Before the statement, we recall the definition of the Eulerian numbers \cite[Section 1.3]{Eulerian}: if $k,n\geq 0$ are integers, then we let
\begin{displaymath}
    \left\langle\begin{array}{c}n\\ k\end{array}\right\rangle=\text{number of permutations of } \lbrace 1,\ldots,n\rbrace\ \text{with}\ k\ \text{descents},
\end{displaymath}
where, for a permutation $\tau$, a descent is an integer $1\leq i<n$ with $\tau(i)>\tau(i+1)$. We shall need the following equivalent representation:
\begin{equation}\label{eq:expansion-Euler-numer}
    \left\langle\begin{array}{c}n\\ k\end{array}\right\rangle=\sum_{j=0}^{k}(-1)^{j}\binom{n+1}{j}(k-j+1)^{n}.
\end{equation}
We refer to \cite[Corollary 1.3]{Eulerian} for a proof. Finally, for integers $p,n\geq 0$, we introduce the following sum of consecutive Eulerian numbers:
\begin{equation}\label{eq:Yoshikawa-coefficient}    
    \Upsilon(n,p)=\sum_{k=0}^{p}\Euler{n}{k}.
\end{equation}
In combinatorial terms, this counts the number of permutations on $\lbrace 1,\ldots, n\rbrace$ having at most $p$ descents. 

We use the conventions $\Euler{0}{0}=1$, $\Euler{n}{k}=0$ for $n\geq 1$ and $k\notin\{0,\ldots,n-1\}$, and $\Upsilon(n,p)=0$ for $p<0$. For $n\geq 1$, the symmetry
\begin{equation}\label{eq:eulernumberssymmetry}
    \Euler{n}{k}=\Euler{n}{n-k-1}
\end{equation}
yields the relationship
\begin{equation}\label{eq:Eulerrelation}
    \Upsilon(n,p)+\Upsilon(n,n-p-2)=n! 
\end{equation}

\subsubsection{} With the above in mind, we have:
\begin{proposition}\label{prop:asymp-Quillen}
Let the notation and assumptions be as above, in particular that $X_0$ has isolated singularities. Let $\sigma$ be a nowhere-vanishing local section of the K\"ahler extension $\lambda(\widetilde{\Omega}_{X/\DBbb}^p)$. Then
    \begin{displaymath}
        \log\|\sigma\|_{Q}^{2}=(-1)^{n-p}\frac{\Upsilon(n+2,p)}{(n+2)!}\mu\log|t|^{2}+\hbox{continuous},
    \end{displaymath}
    as $t\to 0$, where $\mu$ is the total Milnor number of the singular fiber $X_{0}$.
\end{proposition}

\begin{proof}

Suppose first that $f$ arises as the restriction of a projective morphism over an algebraic curve. Then one can apply Yoshikawa's theorem on the singularity of Quillen metrics in \cite{yoshikawa}. Reasoning as in \cite[Theorem 5.11]{FLY} one finds that the logarithmic term has a coefficient as below:
\begin{equation}\label{eq:logOmegap}
    \left(\int_{E}c_{1}(\Ocal(-E)|_{E})^{n}\right)\cdot \sum_{j=0}^p (-1)^{p-j} {n+1 \choose j}\cdot \text{coefficient of}\ X^{n+2}\ \text{in}\  \left(e^{-(p-j+1)X}-e^{-(p-j)X}\right).
\end{equation}
Here $E$ is the exceptional divisor of the blowup of the Jacobian ideal of $f$. The degree of $c_{1}(\Ocal(-E)|_{E})^{n}$ over $E$ is the multiplicity of the Jacobian ideal $\Jcal$ of $f$, which is exactly the sum of Milnor numbers. See \cite[Section 4.3]{Fulton}, and in particular Example 4.3.1 and Example 4.3.5 (c). Then the expression \eqref{eq:logOmegap} reduces to  \begin{equation}\label{eq:complicated-yoshikawa-coefficient}
    (-1)^{n-p}\frac{1}{(n+2)!}\mu\sum_{j=0}^{p}(-1)^{j}\binom{n+1}{j}\left((p-j+1)^{n+2}-(p-j)^{n+2}\right).
\end{equation}
Define
\begin{displaymath}
    A(n+2,p)=\sum_{j=0}^{p}(-1)^{j}\binom{n+1}{j}\left((p-j+1)^{n+2}-(p-j)^{n+2}\right)
\end{displaymath}
Using Pascal's identity for binomial coefficients and using the expression for the Eulerian numbers in \eqref{eq:expansion-Euler-numer}, one straightforwardly verifies that
\begin{displaymath}
     \left\langle\begin{array}{c}n+2\\ k\end{array}\right\rangle = A(n+2,k)-A(n+2,k-1).
\end{displaymath}
Taking the sum over $k=0,\ldots, p$ shows that in fact  $A(n+2,p) = \Upsilon(n+2,p)$.

If $f$ does not arise as the restriction of a family over an algebraic base, we can still apply a version of Yoshikawa's theorem in the case of isolated singularities \cite[Main Theorem]{yoshikawa2}, in which case the algebraicity condition is not needed. One can then infer that the logarithm of the Quillen metric has the stated logarithmic asymptotics as desired, with a coefficient of the form $C(n,p)\cdot\mu$, where $C(n,p)$ is a universal constant depending only on $n$ and $p$. It is then enough to evaluate $C(n,p)$ by testing on a family of relative dimension $n$, extending over an algebraic base, and with a unique ordinary double point in $X_0$. Such families exist, by Brieskorn's globalization argument \cite[Section 1.1]{Brieskorn}. The discussion above applies to this family, and shows that $C(n,p)=(-1)^{n-p}\Upsilon(n+2,p)/(n+2)!$.

\end{proof}
We notice that the combinatorial interpretation of $\Upsilon(n+2,p)$ shows this number is a strictly positive integer. This is not clear from the equivalent sum appearing in \eqref{eq:complicated-yoshikawa-coefficient}, obtained in \cite[Theorem 5.11]{FLY}.

\subsubsection{}

In this subsection we will recall the asymptotics of the $L^2$-norm of a trivializing section $\sigma$ of $\lambda(\widetilde{\Omega}_{X/\DBbb}^p)$. This follows the discussion of the analogous results in \cite{cdg2}. We do not need to suppose that $X_0$ has isolated singularities. We adopt the notation in Section \ref{section:Hodgeinvariants} and Section \ref{sec:comparison-Kahler}. 

We suppose that $X$ is endowed with a K\"ahler form. If $\sigma'$ is a trivializing section of $\lambda(\Omega_{Y/S}^{p}(\log))$, the result in \cite[Theorem 4.4]{cdg2} implies that 
\begin{equation}
\log\|\sigma'\|^2_{L^2}=\alpha^{p}\log|t|^2+\beta^{p}\log\log|t|^{-1} + O(1).
\end{equation}
Here $\alpha^p$ is defined in terms of spectral numbers as in \eqref{eq:alphap}, and 
\begin{equation}\label{eq:betap}
    \beta^{p}=\sum_{q=0}^n (-1)^q \beta^{p,q}, \quad \beta^{p,q}=\sum_{r=-(p+q)}^{p+q} r \dim \Gr^p_F \Gr^W_{p+q+r} H^{p+q}_{\lim}(X_\infty).
\end{equation}
This is an extension of the asymptotics of the $L^2$ norm studied by Schmid \cite{schmid}, which incorporates the contribution from the monodromy eigenvalues. We note that Schmid's results in \cite{schmid}, and consequently \cite[Theorem 4.4]{cdg2}, are stated under an integral polarization condition. This restriction has been relaxed by Sabbah and Schnell \cite{Sabbah-Schnell}, who tackle the general case of variations of polarized \emph{complex} Hodge structures.

On the other hand, we have that $\sigma:=\sigma' \cdot t^{-\mu_p}$ trivializes $\lambda(\widetilde{\Omega}_{X/\DBbb}^{p})$ by \eqref{eq:def-mup}. This purely geometric $\mu_p$ was computed in Theorem \ref{thm:mu-p-general}. We have thus proven the following:
\begin{proposition}\label{prop:L2asymptotic}
    Let $f\colon X \to \DBbb $ be as above, hence not necessarily with $X_0$ having isolated singularities. Assume that $f$ is the restriction of a morphism over an algebraic base. If $\sigma$ is a trivialization of $\lambda(\widetilde{\Omega}_{X/\DBbb}^p)$, then 
    \begin{displaymath}
        \log\|\sigma\|_{L^2}^2=\left(\alpha^{p}+(-1)^{p} \chi(F^{n-p+1}\varphi_f)\right) \log|t|^2+\beta^{p} \log \log|t|^{-1} + O(1),
    \end{displaymath}
    as $t\to 0$. If the singularities in $X_0$ are isolated, this holds without the algebraicity assumption, cf. Lemma \ref{lemma:isolated-removal-algebraicity}.
\end{proposition}

\subsubsection{}

 We now describe the asymptotic behavior of the analytic torsion of $p$-differentials for degenerations with isolated singularities. Recall the notation in \textsection\ref{subsubsec:invariants-milnor} and \textsection\ref{subsubsec:invariants-disjoint-union}.

 \begin{theorem}\label{thm:analytictorsionasymptotics}
 Let $f\colon X\to \DBbb$ be a projective degeneration with isolated singularities in $X_{0}$ as above. Suppose that $X$ is endowed with a K\"ahler metric. Then
\begin{displaymath}
       (-1)^{n-p}\log \tau(X_t, \Omega^{p}_{X_{t}})=\left(\frac{\Upsilon(n+2,p)}{(n+2)!}\mu-\widetilde{s}_{n-p}-\sum_{k=0}^{p-1}s_{n-k}\right) \log|t|^2-\beta^{p,n-p} \log \log|t|^{-1} + O(1),
\end{displaymath}  
as $t\to 0$.
 \end{theorem}
\begin{proof}
This is an immediate combination of \eqref{eq:quillenmetricdef}, Proposition \ref{prop:asymp-Quillen}, Proposition \ref{prop:L2asymptotic} and the definition of the invariants in \textsection\ref{subsubsec:invariants-milnor} and \textsection\ref{subsubsec:invariants-disjoint-union}.
\end{proof}

\section{Around the higher Durfee--Saito conjectures}
In this section, we let $f: (\CBbb^{n+1}, 0) \to (\CBbb, 0)$ be the germ of a holomorphic function, defining an isolated hypersurface singularity at the origin. In the introduction, we stated the higher Durfee--Saito conjectures (Conjecture \ref{conj:higherdurfee}) and a secondary counterpart Conjecture \hyperref[conjecture:B-tilde]{$\widetilde{\mathrm{B}}$}. In this section, we discuss the structural properties of these conjectures.

\subsection{Complements on the higher Durfee--Saito conjectures}

\subsubsection{}
In Conjecture \ref{conj:higherdurfee}, the range for $p$ cannot be improved. In fact, assuming the conjecture, the sense of the inequalities changes beyond the predicted thresholds.

\begin{proposition}\label{prop:durfeesymmetryinequality}
    Suppose that the  higher Durfee--Saito conjecture holds for some $p<\floor{n/2}$, and let $q=n-p-1$. Then 
    \begin{equation}\label{eq:higher-durfee-reverse}
          \frac{\Upsilon(n+1,q)}{(n+1)!} \mu<\sum_{k=0}^{q}s_{n-k}.
    \end{equation}
\end{proposition}
\begin{proof}
    On the one hand, since $\mu = \sum_k s_k$, we have by \eqref{eq:Eulerrelation},
    \begin{displaymath}
        \frac{\Upsilon(n+1,q)}{(n+1)!} \mu = \mu-\frac{\Upsilon(n+1,p)}{(n+1)!} \mu < \sum_{k=0}^n s_k - \sum_{k=0}^{p} s_{n-k}=\sum_{k=0}^{q} s_k.
    \end{displaymath}
On the other hand, by \eqref{eq:steenbrinkineq} we have 
    \begin{equation}\label{eq:steenbrink-inequality}
        \sum_{k=0}^q s_{k}\leq \sum_{k=0}^{q} s_{n-k}.
    \end{equation}
We conclude by concatenating the inequalities.
\end{proof}

The inequalities \eqref{eq:higherDurfeeSaitoconjecture} and \eqref{eq:higher-durfee-reverse} are equivalent in the case that \eqref{eq:steenbrink-inequality} is an equality. This happens exactly whenever $\Gr^{q+1}_{F}H^{n}(M_{f})_{=1}\simeq \Gr^{n-q}_{F}H^{n}(M_{f})_{=1}$ is trivial by Proposition \ref{prop:steenbrinkstyleineq}.\subsubsection{} Conjecture \ref{conj:higherdurfee} and Proposition \ref{prop:durfeesymmetryinequality} actually give inequalities for all $p$, except possibly when $n$  is odd and $p=(n-1)/2.$ In this case, an elementarily argument gives the following unconditional result, whose proof is analogous to that of Proposition \ref{prop:durfeesymmetryinequality} and is omitted:
\begin{proposition}\label{prop:durfeesaitolimitcase}
    Suppose $n$ is odd, and $p=(n-1)/2$. Then 
    \begin{displaymath} \frac{\Upsilon(n+1,(n-1)/2)}{(n+1)!}=\frac{1}{2}
    \end{displaymath} and there is an inequality
    \begin{displaymath}
        \frac{\mu}{2} \leq \sum_{k=0}^{(n-1)/2} s_{n-k},
    \end{displaymath}
    with equality exactly when $\Gr^{(n+1)/2}_F H^n(M)_{=1}$ is trivial.
\end{proposition}

\subsubsection{}\label{subsub:secondaryhigherDurfSaito} We also include the following lemma which states that the secondary higher Durfee--Saito conjecture is in fact a consequence of the Durfee--Saito conjecture. 

\begin{lemma}\label{lemma:BimpliesB-tilde}
If Conjecture \ref{conj:higherdurfee} holds for isolated hypersurface singularities in dimension $n+1$, then Conjecture \hyperref[conjecture:B-tilde]{$\widetilde{\mathrm{B}}$} holds in dimension $n$.
\end{lemma}
\begin{proof}
Choose $N\geq 1$ with $T_s^N=\id$ and consider $g=f+u^{N+1}$. Then $\mu_g=N\mu_f$. For every integer $0\leq p<n/2$, Conjecture \ref{conj:higherdurfee} for $g$, of dimension $n+1$, gives
\[
\sum_{k=0}^{p}s_{n+1-k}(g)<\frac{\Upsilon(n+2,p)}{(n+2)!}N\mu_f.
\]
By Lemma \ref{lemma:ThomSebastianisusp}, the left side is
$N\bigl(\widetilde{s}_{n-p}(f)+\sum_{k=0}^{p-1}s_{n-k}(f)\bigr)$.
Dividing by $N$ proves the required secondary inequality.
\end{proof}

\subsection{Higher Durfee--Saito conjectures and suspensions}\label{subsec:higher-DS-conj-susp}
We next consider the behavior of the higher Durfee--Saito conjectures under double suspensions by squares.

\subsubsection{}  We begin with a combinatorial lemma on the monotonicity of the cumulative Eulerian numbers. 

\begin{lemma}\label{lemma:yoshikawa-eulerian-suspension}
Let $n\geq 3$. Then, the inequality
\begin{displaymath}
    \frac{\Upsilon(n+2,p+1)}{(n+2)!}
    >
    \frac{\Upsilon(n,p)}{n!}
\end{displaymath}
holds for $0\leq p< n/2-1$.
\end{lemma}

\begin{proof}
For integers $0\leq r\leq m$, recall the definition of $\Upsilon(m,r)$ in terms of Eulerian numbers, cf. \eqref{eq:Yoshikawa-coefficient}. We use the standard recurrence for Eulerian numbers \cite[Theorem 1.3]{Eulerian}
\begin{equation}\label{eq:eulerian-recurrence}
\left\langle {m+1\atop k}\right\rangle
=
(m+1-k)\left\langle {m\atop k-1}\right\rangle
+
(k+1)\left\langle {m\atop k}\right\rangle ,
\end{equation}
with the convention that \(\left\langle {m\atop k}\right\rangle=0\) if \(k<0\) or \(k\geq m\).
Summing this identity over \(0\leq k\leq r\), we find
\[
\Upsilon(m+1,r)
=
(m+1)\Upsilon(m,r-1)
+
(r+1)\left\langle {m\atop r}\right\rangle .
\]
We now compare \(\Upsilon(n+2,r)\) with \(\Upsilon(n,r-1)\). Applying the preceding identity once gives
\[
\Upsilon(n+2,r)
=
(n+2)\Upsilon(n+1,r-1)
+
(r+1)\left\langle {n+1\atop r}\right\rangle .
\]
Applying it again to \(\Upsilon(n+1,r-1)\), and using the recurrence \eqref{eq:eulerian-recurrence} for
\(\left\langle {n+1\atop r}\right\rangle\), we obtain

\begin{equation}\label{eq:double-recurrence-Upsilon}
\begin{split}
    \Upsilon(n+2,r)
    =&
    (n+2)(n+1)\Upsilon(n,r-2) + r(n+2)\left\langle {n\atop r-1}\right\rangle  \\
    &+(r+1)\left((n+1-r)\left\langle {n\atop r-1}\right\rangle+(r+1)\left\langle {n\atop r}\right\rangle\right).
\end{split}
\end{equation}
Since by definition of $\Upsilon$ we have
\[
\Upsilon(n,r-1)
=
\Upsilon(n,r-2)
+
\left\langle {n\atop r-1}\right\rangle ,
\]
we can rewrite \eqref{eq:double-recurrence-Upsilon} as
\[
\Upsilon(n+2,r)
=
(n+2)(n+1)\Upsilon(n,r-1)
+
\Delta(n,r),
\]
where
\[
\Delta(n,r)
=
(r+1)^2\left\langle {n\atop r}\right\rangle
-
(n-r+1)^2\left\langle {n\atop r-1}\right\rangle .
\]

We next prove by induction that
\[
    \Delta(n,r)\geq 0
    \quad\text{for}\quad 2r\leq n,
\]
with equality only in the boundary case \(n=2r\).

For \(r=0\), one has
\[
\Delta(n,0)=\left\langle {n\atop 0}\right\rangle=1,
\]
so the assertion is clear. Thus, from now on we may assume $r\geq 1$. For the limit case \(n=2r\), equality follows from the symmetry \eqref{eq:eulernumberssymmetry}, 
since then $r-1=2r-r-1$ so that
\[
\left\langle {n\atop r}\right\rangle
=
\left\langle {n\atop r-1}\right\rangle.
\]

For the induction step, a tedious but straightforward computation from the recurrence \eqref{eq:eulerian-recurrence} gives
\begin{displaymath}
\begin{split}
\Delta(n+1,r)=&(r+1)\Delta(n,r)+(n-r+2)\Delta(n,r-1)  \\
&+(n+2)(n-2r+1)\left\langle {n\atop r-1}\right\rangle.
\end{split}
\end{displaymath}
When \(2r\leq n\), the first two terms are non-negative by induction, while the last term is
positive since we are assuming $r\geq 1$. Hence
\[
\Delta(n+1,r)>0
\]
whenever \(2r<n+1\). This proves the claim.

We finally specialize to \(r=p+1\). The assumption \(p<n/2-1\) is equivalent to
\[
2(p+1)<n,
\]
and therefore \(\Delta(n,p+1)>0\). The identity above gives
\[
\Upsilon(n+2,p+1)
>
(n+2)(n+1)\Upsilon(n,p).
\]
Dividing by \((n+2)!\), we obtain
\[
\frac{\Upsilon(n+2,p+1)}{(n+2)!}
>
\frac{\Upsilon(n,p)}{n!},
\]
as required.

\end{proof}

\begin{corollary}\label{cor:sec-DS-simple-cases}
Let $f\colon (\CBbb^{n+1},0)\to (\CBbb,0)$ be the germ of a holomorphic function, defining an isolated hypersurface singularity at the origin. Define $g=f+u^2+v^2$, where $u,v$ are auxiliary variables. If the higher (resp. secondary higher) Durfee--Saito conjecture holds for $f$ and some $p<\floor{n/2}$ (resp. $p<n/2$), then the higher (resp. secondary higher) Durfee--Saito conjecture holds for $g$ and $p+1$.  
\end{corollary}

\begin{proof}
The statement follows immediately from the lemma and the Thom--Sebastiani properties recalled in \textsection\ref{subsec:TS-properties}. Precisely, $\mu(g)=\mu(f)$, $s_{k+1}(g)=s_{k}(f)$, $\widetilde{s}_{k+1}(g)=\widetilde{s}_{k}(f)$. 
\end{proof}

\subsubsection{} For later use, we record the following direct consequences of Lemma \ref{lemma:yoshikawa-eulerian-suspension}, the original Durfee--Saito conjecture for quasi-homogeneous singularities \cite{Yau-Zhang}, and the secondary counterpart for the spectral genus addressed in \cite[Theorem B]{DURF1}.

\begin{corollary}\label{cor:cases-higher-durfee}
\begin{enumerate} 
    \item Conjecture \ref{conj:higherdurfee} holds for suspensions of the following types of singularities, by an even number of squares:
    \begin{enumerate}
     \item Isolated plane curve singularities.
     \item Isolated quasi-homogeneous singularities of dimension 2 or 3.
    \end{enumerate}
    \item Conjecture \hyperref[conjecture:B-tilde]{$\widetilde{\mathrm{B}}$} holds for suspensions of the following types of singularities, by an even number of squares:
    \begin{enumerate}
    \item Isolated plane curve singularities.
    \item Isolated quasi-homogeneous surface singularities.
    \end{enumerate}
    In particular, the secondary higher Durfee--Saito conjectures hold for $ADE$ singularities.
\end{enumerate}
\end{corollary}

\subsection{Spectral measures and asymptotic sharpness}\label{label:asymptotic-sharpness}
For a germ $f$ as above, recall that we introduced the spectral measure $\nu_f$, as well as the Irwin--Hall measure $\nu_{\IH_n}$, in the introduction \textsection\ref{subsec:spectral-measures-intro}. We also recall that we expressed Conjecture \ref{conj:higherdurfee} and Conjecture \hyperref[conjecture:B-tilde]{$\widetilde{\mathrm{B}}$} in a measure-theoretic language, see \eqref{eq:conj-B-measure}--\eqref{eq:conj-B-tilde-measure}. In this subsection we rely on this interpretation to prove that our conjectures are sharp. 

\subsubsection{} In \cite{KSaito:distribution} K. Saito considers families of functions $\lbrace f_{m}\rbrace_{m}$ for which the spectrum equidistributes with respect to the Irwin--Hall probability measure. Precisely, the convergence is understood as the locally uniform convergence of the Fourier transforms $\widehat{\nu}_{f_{m}}$ to $\widehat{\nu}_{\IH_n}$. This equidistribution property is established for Brieskorn--Pham singularities with exponents tending to $+\infty$, and more generally for quasi-homogeneous singularities with weights tending to $0$, cf. \cite[(2.4) \& (3.7)]{KSaito:distribution}. For other instances of equidistribution, see \cite{Almiron-Schulze}. 

%\subsubsection{}  In \cite[(2.5)--(2.8)]{KSaito:distribution}, Saito also asks to what extent one can bound $\nu_{f}$ by $\nu_{\IH_{n}}$. Our conjectures are related to this question as follows. 
%
%The  Durfee--Saito conjecture for $f$ can be recast as the inequality
%\begin{displaymath}
%    \int_{\RBbb}\mathbf{1}_{[0,p+1]}d\nu_{\IH_{n}}>\int_{\RBbb}\mathbf{1}_{[0,p+1]}d\nu_{f},\quad p<\floor{n/2}.
%\end{displaymath}
%The left hand side indeed equals $\Upsilon(n+1,p)/(n+1)!$. This is immediate from the computations in \cite[pp. 203--205]{KSaito:distribution}, where the relationship with Eulerian numbers was not observed. In fact, this is a classical property of Eulerian numbers, cf. \cite{Euleriannumbers}. The right hand side equals $\mu^{-1}\sum_{k=0}^{p}s_{n-k}$ by \eqref{eq:def-sk-bis}.
%
%Similarly, the secondary counterpart can be written as
%\begin{equation}\label{eq:sec-ineq-measures}
%    \int_\RBbb g_{p}d\nu_{\IH_{n}}>\int_{\RBbb}g_{p}d\nu_{f},\quad p<n/2,
%\end{equation}
%where $g_{p}$ is the continuous function defined by
%\begin{equation}\label{eq:function-varphi}
%    g_{p}(t)=(p+1-t)_{+}-(p-t)_{+}=
%    \begin{cases}
%        1   &\text{if }\ t\leq p,\\
%        p+1-t &\text{if }\  p<t\leq p+1,\\
%        0   &\text{if }\ t>p+1.
%    \end{cases}
%\end{equation}

\subsubsection{} The asymptotic sharpness of our conjectures is an application of these equidistribution results. 

\begin{proposition}\label{prop:asymptotic-sharpness}
Let $f_{m}\colon (\CBbb^{n+1},0)\to (\CBbb,0)$ be a sequence of germs of holomorphic functions, defining quasi-homogeneous isolated singularities at the origin, with weights $w_{i}(m)\to 0$ as $m\to\infty$. Then, 
\begin{displaymath}
 \frac{1}{\mu(f_{m})}\sum_{k=0}^{p}s_{n-k}(f_{m})\to \frac{\Upsilon(n+1,p)}{(n+1)!} \quad\text{for }\ p<\floor{n/2}
\end{displaymath}
and
\begin{displaymath}
    \frac{1}{\mu(f_{m})}\left( \widetilde{s}_{n-p}(f_{m}) + \sum_{k=0}^{p-1} s_{n-k}(f_{m})\right)\to \frac{\Upsilon(n+2,p)}{(n+2)!} \quad\text{for }\ p<n/2,
\end{displaymath}
as $m\to\infty$. 
\end{proposition}
\begin{proof}
The argument is standard, but we provide it for completeness.

In \cite{KSaito:distribution}, K. Saito proves the locally uniform convergence of Fourier transforms $\widehat{\nu}_{f_{m}}\to\widehat{\nu}_{\IH_n}$. L\'evy's continuity theorem \cite[Theorem 15, Chapter 14]{Fristedt-Gray} entails weak convergence of measures, and hence the convergence of the integrals of bounded continuous functions. Applying this to $g_p$ defined in \eqref{eq:function-gp-intro}, we readily obtain the asymptotic sharpness of the secondary conjectures. For the higher Durfee--Saito conjectures, L\'evy's theorem does not apply to the characteristic function $\mathbf{1}_{[0,p+1]}$, since it is not continuous. However, since $\nu_{\IH_{n}}$ is absolutely continuous, it does not charge the boundary of the interval. Then, by the Portmanteau Theorem \cite[Theorem 2.1]{Billingsley}, we have the convergence $\nu_{f_{m}}([0,p+1])\to \nu_{\IH_{n}}([0,p+1])$, as required.
\end{proof}

\subsection{Convex dominance in the case of quasi-homogeneous singularities}

The measure interpretation of the secondary higher Durfee--Saito conjecture from \eqref{eq:conj-B-tilde-measure} motivates the question of comparing the spectral measure with the Irwin--Hall measure. There are several possible orders between measures, and here we consider the convex order, $\nu \preceq_{\mathrm{cx}} \mu$. We will apply it in the study of the sign of $\kappa_f$ in Section \ref{sec:obstructions}.

\subsubsection{} More precisely, we prove the following proposition: 

\begin{proposition}\label{prop:convex-order-QH}
    Let $f\colon (\CBbb^{n+1},0)\to (\CBbb,0)$ be a quasi-homogeneous function defining an isolated singularity at the origin. Then, for every convex function $\varphi\colon \RBbb\to\RBbb$, we have
    \begin{displaymath}
        \int_{\RBbb}\varphi d \nu_{f} \leq \int_{\RBbb}\varphi d \nu_{\IH_{n}}.
    \end{displaymath}
    Hence, $\nu_{f}\preceq_{\mathrm{cx}}\nu_{\IH_{n}}$.
\end{proposition}

\begin{proof}
We denote by $w_{i}\in (0,1)$ the weights of $f$. We first introduce the probability measure
\begin{displaymath}
  \eta:=
  \mathop{*}_{i=0}^{n} \eta_{w_i},
\end{displaymath}
Here $\eta_{w_i}=U[-w_i/2,w_i/2]$ is the uniform probability measure on the indicated interval, and $\ast$ denotes convolution. We will begin by identifying $\nu_{f}\ast\eta$ with a convolution of uniform measures, by a Fourier transform argument. 

As observed by K. Saito \cite[(3.7)]{KSaito:distribution}, the spectral measure and the spectral polynomial of $f$ are related via Fourier duality. Specifically, using the expression for $\operatorname{Sp}_f$ in \eqref{eq:qhomspectrum}, we find that
\begin{equation}\label{eq:fourier-transform-nu-f}
  \widehat{\nu}_f(t) = \frac{1}{\mu}\operatorname{Sp}_f(e^{it}) = \prod_{i=0}^n e^{it/2}\frac{w_i}{1-w_i} \frac{\sin((1-w_i)t/2)}{\sin(w_i t/2)}.
\end{equation}
On the other hand, for uniform measures, we have
\begin{equation}\label{eq:fourier-transform-uniform}
  \widehat{U[a,b]}(t)=e^{it(a+b)/2}\frac{\sin((b-a)t/2)}{(b-a)t/2}.
\end{equation}
Combining \eqref{eq:fourier-transform-nu-f} and \eqref{eq:fourier-transform-uniform}, we easily check the identity
\begin{displaymath}
  \widehat{\nu_f*\eta}(t) = \prod_{i=0}^{n}\widehat{\theta_{w_i}}(t),
\end{displaymath}
where $\theta_{w_i}  = U[w_i/2, 1-w_i/2].$ By the inverse Fourier transform, we conclude
\begin{equation}\label{eq:equalityofsomeconvolutions}
  \nu_f*\eta=\mathop{*}_{i=0}^{n}
  \theta_{w_i}=: \theta.
\end{equation}

Next, we claim that
\begin{equation}\label{eq:nufpreceqnufeta}
  \nu_f\preceq_{\mathrm{cx}}\nu_f*\eta.
\end{equation}
Indeed, if \(\varphi\colon\mathbb R\to\mathbb R\) is a convex function, then using Jensen's inequality and that $\eta$ has zero expectation, we find
\begin{displaymath}
  \varphi(x)=\varphi\!\left(x+\int_{\mathbb R}y\,d\eta(y)\right)
   \leq \int_{\mathbb R}\varphi(x+y)\,d\eta(y)
\end{displaymath}
Integrating with respect to \(\nu_f\), we obtain
\begin{displaymath}
  \int_{\mathbb R}\varphi\,d\nu_f \leq 
  \int_{\mathbb R}\int_{\mathbb R}
  \varphi(x+y)\,d\eta(y)\,d\nu_f(x)=\int_{\mathbb R}\varphi\,d(\nu_f*\eta),
\end{displaymath}
thus proving the claim.

To conclude, by \eqref{eq:nufpreceqnufeta} and \eqref{eq:equalityofsomeconvolutions} it suffices to prove that 
\begin{equation}\label{eq:convex-order-eta-prime-nu-infty}
 \theta \preceq_{\mathrm{cx}}  \nu_{\IH_{n}}.
\end{equation}
Since the convex order is preserved by convolution \cite[Theorem 3.A.12 (d)]{Stochastic-orders}, it is enough to see that
\begin{displaymath}
    U[w/2, 1-w/2]\preceq_{\mathrm{cx}} U[0,1], 
\end{displaymath}
for any $0<w<1$. For this, let $\varphi$ be a convex function. We perform the change of variables 
\begin{displaymath}
    y = (1-w) t + w \frac{1}{2}
\end{displaymath}
from $y\in [w/2, 1-w/2]$ to $t\in [0,1]$. Then, by convexity of $\varphi$
\begin{displaymath}
    \varphi(y) \leq  (1-w) \varphi(t) + w\ \varphi\left(\frac{1}{2}\right). 
\end{displaymath}
Integrating, we find that
\begin{equation}\label{eq:convexinequality}
    \frac{1}{1-w} \int_{w/2}^{1-w/2} \varphi(y) dy \leq (1-w) \int_0^1 \varphi(t) dt + w \
    \varphi\left(\frac{1}{2}\right).
\end{equation}
On the other hand, by Jensen's inequality for convex functions, we have that $\varphi\left(\frac{1}{2}\right)\leq \int_0^1 \varphi(t) dt$. Inserting this into \eqref{eq:convexinequality} proves the convex dominance \eqref{eq:convex-order-eta-prime-nu-infty}, and completes the proof. 
\end{proof}

\subsubsection{} As an application of the previous proposition, we can prove the secondary Durfee--Saito conjecture ($p=0$) for quasi-homogeneous singularities. This was obtained in \cite{DURF1} as a consequence of the work of Stephen T.  Yau and L. Zhang \cite{Yau-Zhang} on the classical Durfee conjecture, and the suspension trick in Lemma \ref{lemma:ThomSebastianisusp}. The approach based on convex dominance is new.

\begin{corollary}\label{cor:Durfee-QH-alternative}
    The secondary Durfee--Saito conjecture is true for quasi-homogeneous isolated singularities of dimension $n\geq 1$. 
\end{corollary}
\begin{proof}
    We know that the general secondary higher Durfee--Saito conjecture can be understood as the inequality \eqref{eq:conj-B-tilde-measure} for the difference $g_p$ in \eqref{eq:function-gp-intro}  of two convex functions. In the case $p=0$, which corresponds to the secondary Durfee--Saito conjecture, on the common support $[0,n+1]$ we have $(-t)_+=0$, so $g_0$ agrees there with the convex function $(1-t)_{+}$. Then, Proposition \ref{prop:convex-order-QH} gives $\widetilde{p}_{g}\leq \mu/(n+2)!$.
    
    To deduce that the inequality is strict, it suffices to verify that one of the inequalities in the proof of Proposition \ref{prop:convex-order-QH} is strict for the function $g_{0}$. Now, for this, we notice that the inequality \eqref{eq:convex-order-eta-prime-nu-infty} is a sequence of inequalities 
    \begin{equation}\label{eq:sequence-inequalities-measures}
        \mathop{*}_{i=0}^n \theta_{w_i} \preceq_{\mathrm{cx}} \mathop{*}_{i=1}^n \theta_{w_i} * U[0,1] \preceq_{\mathrm{cx}} \ldots \preceq_{\mathrm{cx}} \theta_{w_{n}} * U[0,1]^{*n}\preceq_{\mathrm{cx}} U[0,1]^{*(n+1)} =\nu_{\IH_{n}}.
    \end{equation}
    Integrating against $g_0(t)=(1-t)_+$, from the first $n$ inequalities we derive
    \begin{displaymath}
        \int_{\RBbb}(1-t)_{+} d\left(\mathop{*}_{i=0}^n \theta_{w_i}\right)(t)\leq \int_{\RBbb} (1-t)_+  d (\theta_{w_{n}} * U[0,1]^{*n})(t) = \int_{\RBbb} \Psi(x) d \theta_{w_{n}} (x),
    \end{displaymath}
    where 
    \begin{displaymath}
        \Psi(x)=\int_{\RBbb} (1-(x+y))_+ d U[0,1]^{*n}(y) = \frac{(1-x)^{n+1}}{(n+1)!},\qquad x\in[0,1].
    \end{displaymath}
    This function is strictly convex on $[0,1]$, and the last inequality in \eqref{eq:sequence-inequalities-measures} yields 
    \begin{displaymath}
        \int_{\RBbb}\Psi(x) d\theta_{w_{n}} < \int_{\RBbb}\Psi(x) dU[0,1],
    \end{displaymath}
    since Jensen's inequality $\Psi\left(\frac12\right)\leq \int_0^1 \Psi(x) dx $, as applied right after \eqref{eq:convexinequality}, is actually strict.
\end{proof}

To conclude this subsection, we note that the corollary above does not generalize to prove the higher Durfee--Saito conjectures, since the functions $g_p$ in \eqref{eq:function-gp-intro} are not convex for $1\leq p\leq n$. We also remark that we are not aware of any non-quasi-homogeneous singularity for which the analogue of Proposition \ref{prop:convex-order-QH} is false.

\section{An obstruction to smooth fillings of Calabi--Yau families}\label{sec:obstructionsmoothfillings}

In this section, we complete the discussion on the BCOV invariants and the smooth filling problem from the introduction. In particular, in Theorem \ref{theorem:nosmoothfilling}, we provide a numerical criterion for the smooth filling problem of one-parameter families of Calabi--Yau manifolds. Here, a Calabi--Yau manifold means a compact K\"ahler manifold with trivial canonical bundle. This criterion is based on the asymptotic behavior of the BCOV invariant of a degeneration of Calabi--Yau varieties, and it involves the Hodge theoretical invariants of the singularities discussed in Section \ref{section:Hodgeinvariants}.

\subsection{Smooth fillings}
We first give a formal definition of a smooth filling. 

\begin{definition}
Let $f\colon X^\times\to\DBbb^\times$ be a proper holomorphic submersion with  Calabi--Yau fibers.
\begin{enumerate}
    \item We say that $f$ admits a smooth filling if, possibly after a ramified base change $t\mapsto t^m$, $f$ extends to a proper submersion $X\to\DBbb$ with  Calabi--Yau fibers.
    \item We say that $f$ admits a birational smooth filling if, possibly after a ramified base change $t\mapsto t^m$, there exists a proper submersion $Y\to\DBbb$, with  Calabi--Yau fibers, such that $Y_t$ is bimeromorphic to $X_t$ for $t\neq 0$. 
\end{enumerate}
If $X^{\times}\to\DBbb^{\times}$ arises as the restriction of a holomorphic map $X\to\DBbb$ of complex analytic spaces, we will also say that $X\to\DBbb$ admits a smooth (resp. birational smooth) filling if $X^{\times}\to\DBbb^{\times}$ does.
\end{definition}

In the definition of smooth filling, the triviality of the canonical bundle of $X_0$ is actually automatic. Indeed, the coherent sheaf $f_\ast K_{X/\DBbb}$ is necessarily a line bundle, and the evaluation map $f^\ast f_\ast K_{X/\DBbb} \to K_{X/\DBbb}$ is an isomorphism outside of the origin. On the one hand, the divisor of this map is a multiple of the central fiber, since the latter is necessarily connected. On the other hand, the divisor of this map cannot contain a multiple of the central fiber, as noted in \cite[\textsection 2.1]{cdg}, and it is hence trivial. It follows that $K_{X/\DBbb} |_{X_0} \simeq K_{X_0}$ is trivial too. The K\"ahler condition on $X_0$ needs however to be imposed, for it is in general not true that the K\"ahler condition is preserved under specialization, cf. \cite{Hironaka:counterexample}.

In the definition of birational smooth filling, we only impose a fiberwise bimeromorphic condition, but we do not impose any relationship between the bimeromorphic modifications of different fibers.   
\subsection{The BCOV invariant and its variation }
 
 We next recall that the BCOV invariant is defined, in \cite{cdg2}, as
\begin{equation}\label{eq:def-bcov-invariant}
    \tau_{\mathrm{BCOV}}(X)=\frac{A(X,\omega)}{B(X,\omega)}\prod_{p}\tau(\Omega_{X}^{p})^{(-1)^{p}p}.
\end{equation}
It is a positive real number. Here $A(X, \omega)$ and $B(X, \omega)$ are certain normalizing factors, introduced to render the expression independent of the auxiliary K\"ahler form $\omega$. The factor $A(X,\omega)$ is a positive real number; when the K\"ahler class is integral, the covolume factor $B(X,\omega)$ is rational.

\subsubsection{} This extends to families. Let $f\colon X\to S$ be a proper submersion of complex manifolds, with Calabi--Yau fibers. It follows from \cite[Section 5]{cdg2} that if $X$ admits a K\"ahler metric, or more generally $f$ is a K\"ahler morphism, the BCOV invariant is a smooth function on $S$. More generally, we have the following statement, suggested to us by Ken--Ichi Yoshikawa: \begin{proposition}\label{prop:bcovinvsmooth}
Let $f\colon X\to S$ be a proper submersion of complex manifolds, with  Calabi--Yau fibers. Then, the logarithm of the BCOV invariant of the fibers defines a smooth function on $S$. 
\end{proposition}

\begin{proof}
        Let $0\in S$ be fixed. Since the fiber $X_0$ is K\"ahler, a result by Kodaira--Spencer \cite[Theorem 15]{Kodaira-Spencer} ensures that the nearby fibers also admit a K\"ahler metric. This is already assumed in the proposition, but in fact their argument proves that, after possibly shrinking $S$, $T_{X/S}$ admits a $\Ccal^{\infty}$ metric whose restriction to $X_s$ is K\"ahler, for $s\in S$. We denote the corresponding K\"ahler form by $\omega_s$. We compute the BCOV invariant of the fibers using these metrics. 
        
        The definition \eqref{eq:def-bcov-invariant} involves three terms. For the normalizing factors $A(X_s, \omega_s)$ and $B(X_s, \omega_s)$, the very definition in \cite[Section 5]{cdg2} shows that they can be expressed as fiber integrals of relative differential forms on $X$, and hence they define smooth functions of $s\in S$. By the definition of analytic torsion  \eqref{def:analytictorsion}, the weighted product of torsions $\tau(\Omega^{p}_{X_{s}})$ is a weighted product of exponentials of expressions of the form $\zeta'_{p,q,s}(0)$ as in  \eqref{def:zetafunction}. We next tackle those. 
        
        Since the fibers are K\"ahler, $\dim \ker \Delta^{p,q}_{\overline{\partial},s}= h^{p,q}(X_s)$ is constant and the eigenvalues of $\Delta^{p,q}_{\overline{\partial},s}$ depend continuously on $s$ by \cite[Theorem 2]{Kodaira-Spencer}. Hence, after possibly shrinking $S$ around $0$, we can find $b>0$ such that $\Delta^{p,q}_{\overline{\partial},s}$ has no eigenvalues $\lambda > 0$ with $\lambda < b$. This means that $\zeta_{p,q,s}$ is also the spectral zeta function built from the eigenvalues strictly larger than $b$. For this function, it follows from \cite[Section g]{Bismut--Freed} that $\zeta_{p,q,s}^{\prime}(0)$ is smooth in $s$, hence concluding the proof.
    \end{proof}

\subsubsection{} The fact that the invariant is moving smoothly in such families gives an obstruction to smooth fillings, as stated in the following proposition. 

\begin{proposition}\label{prop:criterion-smooth-filling}
Let $f\colon X^\times\to\DBbb^\times$ be a proper holomorphic submersion with  Calabi--Yau fibers. Suppose that the function $t\mapsto\log\tau_{\mathrm{BCOV}}(X_t)$ does not extend continuously at $t=0$. Then $f$ does not admit a smooth filling. Moreover, if the fibers of $f$ are projective, then $f$ does not admit a birational smooth filling.
\end{proposition}
\begin{proof}
Let $m\geq 1$. We observe that $t\mapsto\log\tau_{\mathrm{BCOV}}(X_t)$ extends continuously at $0$ if, and only if, $t\mapsto\log\tau_{\mathrm{BCOV}}(X_{t^{m}})$ extends continuously at $0$. Therefore, the first claim follows from the definition of smooth filling and Proposition \ref{prop:bcovinvsmooth}. 

For the second claim, suppose that $f$ admits a birational smooth filling $Y\to \DBbb$. For some fixed $m\geq 1$, the fibers $Y_{t}$ are bimeromorphic to $X_{t^m}$, for $t\neq 0$. Since the latter is projective, we infer that $Y_{t}$ is Moishezon. Since $Y_t$ is K\"ahler by assumption, it is thus projective and birational to $X_{t^{m}}$. By \cite{ZhangFu-birationalBCOV}, the BCOV invariant of projective Calabi--Yau manifolds is a birational invariant, hence $\tau_{\mathrm{BCOV}}(X_{t^{m}})=\tau_{\mathrm{BCOV}}(Y_{t})$. We conclude by Proposition \ref{prop:bcovinvsmooth} applied to $Y\to\DBbb$.
\end{proof}

In the second part of the proposition, we note that the projectivity assumption on the fibers does not in general entail the existence of a relatively ample line bundle on $X^\times$. We refer to \cite{Kollar:Seshadri} for a study of openness of the projectivity condition and for references to counterexamples.

\subsection{The asymptotic behavior of the BCOV invariant}\label{subsec:generasymp}
We apply the previous results on the behavior of analytic torsions of sheaves of differential forms (Section \ref{section:torsion-differentials}) to the study of the BCOV invariant.

\subsubsection{} We begin by recalling, and slightly generalizing, \cite[Theorem 6.5]{cdg2} on the asymptotic behavior of the BCOV invariant for projective degenerations of Calabi--Yau manifolds. 

\begin{lemma}\label{lemma:existence-kappa}
Let $f\colon X\to\DBbb$ be a projective morphism of complex manifolds, such that $X_t$ is Calabi--Yau for $t\neq 0$. Then, if either $f$ arises as the restriction of a morphism of algebraic varieties, or $X_0$ has isolated singularities, we have
\begin{equation}\label{eq:coefficient-kappa}
    \log \bcov{t}  = \kappa_{f}\cdot \log|t|^{2}+o(\log|t|),\quad\text{as}\quad t\to 0,
\end{equation}
for some $\kappa_f\in\QBbb$.
\end{lemma}

\begin{proof}
When $f$ is the restriction of a morphism of algebraic varieties, this is \cite[Theorem 6.5]{cdg2}. If $X_0$ has only isolated singularities, the proof goes along the same lines, using Theorem \ref{thm:analytictorsionasymptotics} above. We omit the details.
\end{proof}

The following is an immediate application of the previous lemma and Proposition \ref{prop:criterion-smooth-filling}.
\begin{theorem}\label{theorem:nosmoothfilling}
Let the assumptions and notation be as in Lemma \ref{lemma:existence-kappa}. Then, if $\kappa_{f} \neq 0$, the family $X\to\DBbb$ does not admit a birational smooth filling.
\end{theorem}

\subsubsection{} Suppose now that $f\colon X\to\DBbb$ is a flat projective morphism of complex manifolds, with Calabi--Yau fibers $X_t$ for $t\neq 0$, but $X_0$ possibly singular. We wish to provide expressions for the coefficient $\kappa_f$ in \eqref{eq:coefficient-kappa}.

In preparation for the statement below, consider the evaluation map $f^\ast f_\ast K_{X/\DBbb} \to K_{X/\DBbb}$. The Calabi--Yau condition ensures this is an isomorphism outside the origin, and we denote by $B$ the divisor of the morphism.

The main result is the following, which generalizes the previously known general formulas in dimension $n=3,4$, treated in \cite[Section 7.3]{cdg2}. It expresses the numerical invariant $\kappa_f$ in terms of a local contribution $\kappa_f^{\mathrm{loc}}$ and a global contribution. The local contribution depends only on the vanishing cycles complex. The global contribution involves the topological Euler characteristic of a general fiber, and the divisor $B$.

\begin{theorem}\label{thm:kappaBCOVgeneral}
Let $f\colon X\to\DBbb$ be a degeneration of Calabi--Yau varieties as above, of dimension $n\geq 2$. Assume furthermore that $f$ arises as the restriction of a projective morphism of algebraic varieties. Then,
\begin{displaymath}
    \kappa_{f}=\kappa_{f}^{\mathrm{loc}}-\left(\frac{\chi(X_{\infty})}{12}\alpha^{0,n}+\frac{(-1)^{n}}{12}\int_{B}c_{n}(\Omega_{X})\right),
\end{displaymath}
where:
\begin{enumerate}
    \item if $n$ is even, then
\begin{displaymath}
    \kappa_{f}^{\mathrm{loc}}=-\frac{n-2}{24}  \chi(\varphi_f) + \sum_{p=0}^{n/2-1} (-1)^p (n-2p) \alpha^p + \sum_{p=1}^{n/2-1}(n-2p) \chi(F^{n-p+1} \varphi_f).
\end{displaymath}
    \item if $n$ is odd, then:
\begin{displaymath}
    \kappa_{f}^{\mathrm{loc}}=-\frac{n+1}{24} \chi(\varphi_f) +\sum_{p=0}^{\floor{n/2}} (-1)^p (n-2p) \alpha^p + \sum_{p=1}^{\floor{n/2}}(n-2p) \chi(F^{n-p+1} \varphi_f).
\end{displaymath}
\end{enumerate}
In particular, if $X_0$ is integral and the singularities are Du Bois, then $\kappa_{f}=\kappa_{f}^{\mathrm{loc}}$ only depends on an analytic neighborhood of the singularities in $X$. 
\end{theorem}
\begin{proof}
The proof goes along the lines of the proof of \cite[Theorem 7.6 \& Theorem 7.13]{cdg2}, and we only review the key points.  

We need to control the expression \eqref{eq:def-bcov-invariant}, which in fact, by its very definition in \cite{cdg2}, naturally can be written as a quotient of two natural metrics $\frac{h_{Q,BCOV}}{h_{L^2, BCOV}}$ on the BCOV bundle 
\begin{displaymath}
    \bigotimes_{p} \lambda(\Omega_{X^{\times}/\DBbb^{\times}}^p)^{(-1)^{p}p}.
\end{displaymath} 
In particular, we give the asymptotics of both metrics on the K\"ahler extension $\bigotimes_{p} \lambda(\widetilde{\Omega}_{X/\DBbb}^p)^{(-1)^{p}p}$, defined as in \textsection \ref{subsec:Kahler-extension}. For the Quillen-BCOV metric, we use the asymptotic formula in \cite[Corollary 4.9]{cdg}. After introducing a K\"ahler form with integral cohomology class on the smooth fibers, the asymptotics of the $L^{2}$-BCOV metric reduce to the asymptotics of the $L^{2}$-metric, since the $B$ factor defines a locally constant function on $\DBbb^{\times}$. For the latter, we isometrically rewrite the K\"ahler extension of the BCOV bundle, by applying Serre duality on the contributions $\lambda(\widetilde{\Omega}^{p}_{X/\DBbb})$ for $p\geq n/2$. The isomorphism on a general fiber does not extend to the whole extension, and the discrepancy between extensions is measured by the vanishing cycles $\chi(\varphi_f)$ by \cite[Proposition 3.3]{cdg2}. Finally, we employ the asymptotics of the $L^2$-metric in Proposition \ref{prop:L2asymptotic}, the new ingredient from this article. The theorem is now an immediate arrangement of these terms.

The locality statement at the end follows since the divisor $B$ is necessarily $0$ if $X_0$ is integral, as argued in \cite[Section 2]{cdg}, and $\alpha^{0,n}=0$ for a Du Bois singularity by a result of M. Saito \cite{Saito:du-Bois}. Indeed, by \cite[Theorem A]{cdg} $\alpha^{0,n}$ equals $1-\mathrm{lct}(X, X_0)$, where $\mathrm{lct}$ stands for the log-canonical threshold of $X_0$ in $X$. The latter is the minimum of 1 and the minimal exponent by \cite[Section 10]{kollar:singpairs}, and Saito's results ensure that the minimal exponent is at least 1. All in all, $\alpha^{0,n}=0$. 

\end{proof}
\subsection{Asymptotics for isolated singularities}\label{subsec:isolasymp}
We next specialize and simplify the asymptotics of the previous subsection for isolated singularities. In this case, we can bypass the algebraicity restriction in Theorem \ref{thm:kappaBCOVgeneral}. 

\subsubsection{}\label{subsubsec:asymptoticsforisolatedsing} First let $f: X \to \DBbb$ be a projective degeneration of Calabi--Yau varieties of dimension $n\geq 2$, with smooth total space $X$, such that $X_0$ has at most isolated singularities. Whenever $f$ arises as the restriction of an algebraic map, a straightforward application of Theorem \ref{thm:kappaBCOVgeneral} also gives the following simplifications:

We obtain
\begin{equation}\label{eq:kappafusefulforlemmaformula}
    \kappa_{f}=(-1)^{n+1}\sum_{2p<n}(n-2p)\left(\frac{\Upsilon(n+2,p)}{(n+2)!}\mu-\widetilde{s}_{n-p}-\sum_{k=0}^{p-1}s_{n-k}\right)-\frac{\chi(X_\infty)}{12}\widetilde{s}_{n} +(-1)^{n} \frac{\mu}{12}.
\end{equation}

In the general case, we can arrive at the same expression using Theorem \ref{thm:analytictorsionasymptotics}. Here we can use Serre duality for the torsion, which gives $\log \tau(X_t, \Omega_{X_t}^p) = (-1)^{n+1} \log \tau(X_t, \Omega_{X_t}^{n-p})$, and the asymptotic behavior of the factor $A$ in the proof of \cite[Proposition 4.2]{cdg}.

\subsubsection{}

In the above expression, it is possible to isolate and study more generally the local part of $\kappa_f$, outside of the Calabi--Yau context, by removing the term $\frac{\chi(X_\infty)}{12}\widetilde{s}_{n}$. This motivates the following definition:
\begin{definition}\label{eq:def-kappa-f-loc}
    Let $f: (\CBbb^{n+1}, 0) \to (\CBbb, 0)$ be the germ of an isolated singularity. We define 
    \begin{equation}\label{eq:kappaflocisolated}
            \kappa_{f}^{\mathrm{loc}}=(-1)^{n+1}\sum_{2p<n}(n-2p)\left(\frac{\Upsilon(n+2,p)}{(n+2)!}\mu-\widetilde{s}_{n-p}-\sum_{k=0}^{p-1}s_{n-k}\right)+(-1)^{n} \frac{\mu}{12}.
    \end{equation}
\end{definition}

We record the following simplifications of this invariant:

    \begin{enumerate}
        \item if $n$ is even, then
    \begin{displaymath}
        \kappa_{f}^{\mathrm{loc}}= \frac{2-n}{24}\mu+\sum_{p=0}^{n/2-1}(n-2p) \widetilde{s}_{n-p}+\sum_{r=2}^{n/2} r(r-1) s_{n/2+r}.
    \end{displaymath}
    \item If $n$ is odd, then
    \begin{displaymath}
        \kappa_{f}^{\mathrm{loc}} = \frac{n+1}{24}\mu-\sum_{p=0}^{\floor{n/2}} (n-2p)\widetilde{s}_{n-p}-\sum_{r=1}^{\floor{n/2}} r^2 s_{\floor{n/2}+r+1}.
    \end{displaymath}
        \end{enumerate}

The only non-direct simplification is that of the constant in front of the Milnor number. This is a universal constant $C(n)$ in the dimension, which can be determined by specializing to the case of an ordinary double point. In this setting, one can find a globalization into an algebraic Calabi--Yau family, e.g. a Lefschetz pencil of Calabi--Yau hypersurfaces in $\PBbb^{n+1}$, and then $\kappa_{f}$ is derived from Theorem B of \cite{cdg2}.

\subsubsection{}\label{subsubsec:varianceandbeta} For the next simplification, we introduce some notation. 

Let $f$ be a germ of an isolated singularity as above, with spectral numbers $\lambda_{i}\in (0,n+1)$. We define its variance by
\begin{displaymath}
    V_{f}=\frac{1}{\mu}\sum_{i=1}^{\mu}\left(\lambda_{i}-\frac{n+1}{2}\right)^{2}.
\end{displaymath}
We also refer to the introduction for the definition of the invariant $\beta_f$, see \eqref{eq:def-beta-intro}. Finally, we define a convex function on $\RBbb$ by the formula
\begin{equation}\label{def:integrandforkappa}
    \phi_{n}(t)=2\sum_{p=0}^{m}(p-t)_{+}+(n-2m)(m+1-t)_{+},\quad m=\left \lfloor \frac{n-1}{2}\right \rfloor.
\end{equation}

\subsubsection{} For the forthcoming statement, recall from \textsection \ref{label:asymptotic-sharpness} the spectral measure $\nu_{f}$ of $f$ and the Irwin--Hall measure $\nu_{\IH_{n}}$.

\begin{theorem}\label{thm:kappa-f-int-phi}
Let the notation be as above. Then:
\begin{enumerate}
    \item There is an integral representation
        \begin{displaymath}
            (-1)^{n+1}\frac{\kappa_{f}^{\mathrm{loc}}}{\mu}=\int_{\RBbb}\phi_{n}(d\nu_{\IH_{n}}-d\nu_{f})-\frac{1}{12}.
        \end{displaymath}
    \item There is a spectral number representation
        \begin{equation}\label{eq:kappa-Vf-betaf}
             (-1)^{n+1}\frac{\kappa_{f}^{\mathrm{loc}}}{\mu}=\frac{n-1}{24}-\frac{1}{2}V_{f}+\frac{1}{2}\beta_{f}.
        \end{equation}
\end{enumerate}
\end{theorem}

\begin{proof}
    Equation \eqref{eq:kappaflocisolated} expresses $\kappa_f^{\mathrm{loc}}$ as a linear combination of higher Durfee--Saito defects. As observed in \eqref{eq:conj-B-tilde-measure}, the latter can be recast in terms of the integration of the functions $g_{p}(t)=(p+1-t)_{+}-(p-t)_{+}$. Explicitly, the result is
    \begin{displaymath}
        (-1)^{n+1}\frac{\kappa_{f}^{\mathrm{loc}}}{\mu}=\sum_{0\leq 2p<n}(n-2p)\int_{\RBbb}g_{p}(d\nu_{\IH_{n}}-d\nu_{f})-\frac{1}{12}.
    \end{displaymath}
    The first sum is a telescoping sum which affords a direct simplification as stated. 

    For the spectral number representation, we claim that we have the functional identity on $[0,n+1]$,
    \begin{equation}\label{eq:functional-equation-phi}
        \phi_{n}(t)+\phi_{n}(n+1-t)=\left(t-\frac{n+1}{2}\right)^{2}+\lbrace t\rbrace (1-\lbrace t\rbrace)-\frac{1}{8}(1+(-1)^{n}).
    \end{equation}
Indeed, by construction, both sides of \eqref{eq:functional-equation-phi} are invariant under the symmetry map $t\mapsto n+1-t$. Second, both sides are piecewise linear with breaks at the integer points, as we see by evaluating on the intervals of the form $[j,j+1]$, for $j=0,\ldots, n$. It is thus enough to show that both sides agree at the integers $j\leq (n+1)/2$. This is an elementary computation that we omit.  

Finally, to conclude \eqref{eq:kappa-Vf-betaf}, we integrate \eqref{eq:functional-equation-phi} against $d\nu_{\IH_{n}}-d\nu_{f}$. To perform the integration, we use that both measures are invariant under $t\mapsto n+1-t$ and that they are both probability measures with the same expectation $(n+1)/2$. For $\nu_{f}$, this amounts to the symmetry of spectral numbers $\lambda_{i}+\lambda_{\mu+1-i}=n+1$, cf. \cite[\textsection 12.1.3]{Peters-Steenbrink}. One also takes into account the evaluation
\begin{equation}\label{eq:variance-continuous}
    \int_{0}^{n+1}\left(t-\frac{n+1}{2}\right)^{2}d\nu_{\IH_{n}}(t)=\frac{n+1}{12}.
\end{equation}
We also use $\int_{\RBbb}\widetilde B_2(t)\,d\nu_{\IH_n}(t)=0$, where we recall that $\widetilde{B}_2(x) = B_2(\{x\})$ . For this, integrate first against one uniform factor and apply the periodicity of $\widetilde B_2$ and $\int_0^1 B_2(t)\,dt=0$.
\end{proof}

\section{Obstructions for special geometries}\label{sec:obstructions}
In this section, we apply the previous analysis on the invariants $\kappa_f$ and $\kappa_{f}^{\mathrm{loc}}$ for various global and local geometries, and we deduce non-existence results of birational smooth fillings of Calabi--Yau families. 

For germs of isolated hypersurface singularities, not necessarily arising from Calabi--Yau families, we propose a positivity conjecture for $\kappa_{f}^{\mathrm{loc}}$ refining Conjecture \ref{conjectureintro-kappa-positive}. This conjecture is supported by extensive computer searches and several positive results addressed below. The case of Brieskorn--Pham singularities will be treated at greater length, and is left to its own section later.

In the case of non-isolated singularities, we study an example of a family of Calabi--Yau 3-folds with finite monodromy and no smooth fillings, due to Cynk and van Straten \cite{CynkStraten}, and we show that their result is also covered by a study of $\kappa_f$. We also provide a new example inspired by the work of Fang, Lu and Yoshikawa \cite{FLY}. 

\subsection{The obstruction for terminal singularities} 
In this subsection we discuss a local form of Conjecture \ref{conjectureintro-kappa-positive} and some simple criteria for terminal singularities useful for our estimates. 

\begin{conjectureA'}\label{conj:localkappa}
Let $f\colon(\CBbb^{n+1},0)\to(\CBbb,0)$ define an isolated hypersurface
singularity, and suppose that $n\geq 3$.  Then, if $f=0$ is terminal, 
    \begin{displaymath}
        (-1)^{n+1}\kappa_f^{\mathrm{loc}}>0.
    \end{displaymath}
\end{conjectureA'}

In the particular case that $f$ arises as the restriction of a Calabi--Yau family, still denoted by $f$, we note that $\kappa_{f}=\kappa_{f}^{\mathrm{loc}}$, since the terminal condition implies the vanishing of $\widetilde{s}_{n}$ in \eqref{eq:kappafusefulforlemmaformula}. Thus Conjecture \hyperref[conj:localkappa]{$\mathrm{A}^{\prime}$} at every singular point implies Conjecture \ref{conjectureintro-kappa-positive} for the family. If there is only one singular point, the two inequalities coincide.  

\subsubsection{}
A source of terminal singularities comes from some sums of the form 
\begin{displaymath}
    f(u_0,\ldots, u_{n}, v_{0},\ldots, v_m)=g(u_0,\ldots, u_{n})+h(v_{0}, \ldots, v_m),
\end{displaymath}
where $g$ and $h$ both define isolated singularities at $0$. If $h$ defines a Du Bois singularity, we will refer to $f$ as a Du Bois-suspended singularity. For example, $h=v_0^2+v_1^2$ would be of this type, in which case we say that $f$ is doubly suspended. $ADE$ singularities are of this form, and they are hence Du Bois-suspended. Similarly, if $h$ defines a rational singularity, then we say that $f$ is a rationally suspended singularity. Since rational singularities are Du Bois, a rationally suspended singularity is in particular Du Bois-suspended. 

\begin{lemma}\label{lemma:1duboissuspended}
    Suppose that $n \geq 1$. Then a Du Bois-suspended singularity is terminal.
\end{lemma}
\begin{proof}
    We notice that if $n \geq 1$, it follows by Bertini that the intersection of $g=0$ by a general hyperplane $H$ through the origin in $\CBbb^{n+1}$ is still an isolated hypersurface singularity, and the smallest spectral number is larger than 0.
    
    By \cite[p. 1372]{Steenbrink:Du-Bois} and \eqref{eq:def-sk-bis}, the condition $s_0(h)=0$  characterizes $h=0$ being Du Bois. See also the proof of Theorem \ref{thm:kappaBCOVgeneral} above. It follows by the description \eqref{eq:dim-Hpq-spectral} that the minimal spectral number of $h$ is at least 1. Because of the Thom--Sebastiani property of spectral numbers recalled in \textsection\ref{subsec:TS-properties}, the smallest spectral number of the restriction of $f$ along $H\times\CBbb^{m+1}$ is the sum of the smallest spectral numbers of $g|_H$ and $h$ and hence bigger than 1.
    
    Again by the description \eqref{eq:def-sk-bis}, we find that $s_{n+m}(f\mid_{H\times\CBbb^{m+1}})=0$. By \cite{MSaitogenus, SteenbrinkMixedAssociated} this means exactly that the singularity $f=0$ restricted to $H\times\CBbb^{m+1}$ is canonical.  
    
    Finally, since the singularities are isolated, it follows by an inversion of adjunction argument, as in \cite[Example 2.2.14]{Prokhorov1999Complements}, that the singularity of $f$ is terminal. 
\end{proof}

\subsubsection{} In the statements below, we denote by $\lambda_{\min}(f)$, or simply $\lambda_{\min}$, the minimal spectral number of the germ of an isolated singularity $f=0$. 

\begin{lemma}\label{lemma:1duboisterminal}
    Suppose that $f: (\CBbb^{n+1}, 0) \to (\CBbb, 0)$ defines the germ of an isolated singularity at the origin, with $\lambda_{\min}(f)>3/2$. Then $f=0$ is terminal.
\end{lemma}
\begin{proof}
If $H$ is a general hyperplane in $\CBbb^{n+1}$ through the origin, then by \cite[Theorem 1.5]{Dirks-Mustata} the minimal spectral numbers of $f$ and $f_{\mid H}$ are related by
\begin{displaymath}
    \lambda_{\mathrm{\min}}(f_{\mid H})\geq  \lambda_{\mathrm{\min}}(f)-\frac{1}{\operatorname{mult}_{0}(f)}.
\end{displaymath}
By assumption, $\lambda_{\mathrm{\min}}(f) > 3/2$, and moreover the multiplicity of $f$ at $0$ is at least 2. Hence, $\lambda_{\mathrm{\min}}(f_{\mid H})>1$, and $f_{\mid H}$ defines a rational singularity. By the inversion of adjunction argument in the proof of the previous lemma, we conclude that $f=0$ has a terminal singularity at the origin.
\end{proof}

\subsubsection{} We record the following direct relation of Conjecture \hyperref[conj:localkappa]{$\mathrm{A}^{\prime}$} with the higher Durfee--Saito conjectures:
\begin{proposition}\label{prop:doubly-suspended-kappa}
Let $f$ be a doubly suspended isolated singularity of dimension $n\geq 3$, of the form $f = g + v_0^2 + v_1^2$, in particular terminal. Suppose that either of the following conditions is satisfied:
\begin{enumerate}
    \item The secondary higher Durfee--Saito conjectures hold for $g$. 
    \item The function $g$ is quasi-homogeneous.
\end{enumerate}
Then Conjecture \hyperref[conj:localkappa]{$\mathrm{A}^{\prime}$} holds.
\end{proposition}

\begin{proof}
First of all, the claim on terminality is an application of Lemma \ref{lemma:1duboissuspended}. Second, we observe that
\begin{equation}\label{eq:kappafdoublesuspended}
   \begin{split}
    \kappa_{f}^{\mathrm{loc}} =& \kappa_g^{\mathrm{loc}} +(-1)^{n+1} \frac{\mu(f)}{12}\\
    =& (-1)^{n+1}\sum_{2p<n-2}(n-2-2p)\left(\frac{\Upsilon(n,p)}{n!}\mu(g)-\widetilde{s}_{n-2-p}(g)-\sum_{k=0}^{p-1}s_{n-2-k}(g)\right).
    \end{split}
\end{equation}
This follows since we have $\mu(f)=\mu(g)$, and by Thom--Sebastiani, $s_{k+1}(f)=s_{k}(g)$ and $\widetilde{s}_{k+1}(f)=\widetilde{s}_{k}(g)$. The result is then a simple manipulation of the expression for $\kappa_{f}^{\mathrm{loc}}$ in \eqref{eq:kappaflocisolated}. The expression \eqref{eq:kappafdoublesuspended} yields the proposition under the assumption of the secondary higher Durfee--Saito conjectures for $g$. 

In the case that $g$ is quasi-homogeneous, the same conclusion holds, by the convex order dominance of Proposition \ref{prop:convex-order-QH}. Indeed, in this case, by Theorem \ref{thm:kappa-f-int-phi}, equation \eqref{eq:kappafdoublesuspended} can be recast 
\begin{displaymath}
    (-1)^{n+1} \kappa_f^{\mathrm{loc}} = (-1)^{n+1} \kappa_g^{\mathrm{loc}} + \frac{\mu(f)}{12} = \mu(f)\int_{\RBbb} \phi_{n-2}  (d\nu_{\IH_{n-2}} - d\nu_g)
\end{displaymath}
where $\phi_{n-2}$ is the convex function defined as in \eqref{def:integrandforkappa}. On the support of the measures, the function $\phi_{n-2}$ contains the term $(1-t)_+$ with a strictly positive coefficient. All the other terms have nonnegative coefficients. Corollary \ref{cor:Durfee-QH-alternative} gives strict positivity for the summand $ (1-t)_+$, while convex dominance gives nonnegativity for the others. Thus, the integral is strictly positive.
\end{proof}
\subsubsection{} In the following proposition, we provide another conditional and unconditional confirmation of Conjecture \hyperref[conj:localkappa]{$\mathrm{A}^{\prime}$}, relating to the Hertling variance conjecture  \cite{Hertling}. We recall that this predicts an inequality
\begin{equation}\label{eq:Hertlinginequality}
    V_f \leq \frac{n+1-2\lambda_{\min}}{12}.
\end{equation}
It is known for quasi-homogeneous singularities. In this case, the inequality \eqref{eq:Hertlinginequality} is in fact an equality \cite[Theorem 7.2]{Hertling}.

\begin{proposition}\label{prop:1-du-bois-kappa}
Suppose that $f: (\CBbb^{n+1}, 0) \to (\CBbb, 0)$ defines an isolated singularity of dimension $n\geq 3$ with minimal spectral number $\lambda_{\mathrm{\min}}>3/2$.  Suppose that either of the following conditions is satisfied:
    \begin{enumerate}
        \item The Hertling variance conjecture holds for $f$.
        \item The function $f$ is quasi-homogeneous.
    \end{enumerate}
    Then Conjecture \hyperref[conj:localkappa]{$\mathrm{A}^{\prime}$} holds.

\end{proposition}
\begin{proof}
For a general hyperplane $H$, the inequality used in Lemma \ref{lemma:1duboisterminal} gives $\lambda_{\mathrm{\min}}(f_{\mid H})\geq\lambda_{\mathrm{\min}}(f)-1/2>1$. Hence the general hyperplane section is rational, and the same inversion of adjunction argument shows that the singularity is terminal.  Combining the inequality \eqref{eq:Hertlinginequality} in the Hertling variance conjecture with \eqref{eq:kappa-Vf-betaf} in Theorem \ref{thm:kappa-f-int-phi}, we find
\begin{equation}\label{eq:estimate-kappa-after-hertling}
    (-1)^{n+1}\frac{\kappa_{f}^{\mathrm{loc}}}{\mu}\geq\frac{\lambda_{\mathrm{\min}}-1+6\beta_{f}}{12}.
\end{equation}
The Bernoulli polynomial $B_2(x)$ satisfies $B_2(x) \geq -\frac{1}{12}$ on $[0,1]$, so from the definition of $\beta_f$ in \eqref{eq:def-beta-intro} we get the crude estimate $\beta_f \geq -\frac{1}{12}$. Inserting this bound in \eqref{eq:estimate-kappa-after-hertling}, we obtain 
\begin{displaymath}
    (-1)^{n+1}\frac{\kappa_{f}^{\mathrm{loc}}}{\mu}\geq \frac{\lambda_{\mathrm{min}}-3/2}{12}>0.
\end{displaymath}

The second item follows from the first one, since for quasi-homogeneous singularities the variance conjecture is known, as recalled above.

\end{proof}

\subsubsection{}
We provide a variant of the previous proposition for a different class of singularities. 
\begin{proposition}\label{prop:rationally-suspended}
Suppose that $f: (\CBbb^{n+m+2}, 0) \to (\CBbb, 0)$ is the germ of an isolated singularity, of the form $g+h$ where:
        \begin{enumerate}
        \item $g$ defines an isolated singularity of dimension $m$ and satisfies the Hertling variance conjecture.
        \item $h$ is quasi-homogeneous and defines an isolated singularity of dimension $n \geq 0$.
        \end{enumerate}
Then the following inequality holds:
        \begin{displaymath}
        (-1)^{n+m+2}\frac{\kappa_{f}^{\mathrm{loc}}}{\mu(f)} \geq \frac{\lambda_{\min}(g)-1}{12}.
        \end{displaymath}
In particular, if $f$ is quasi-homogeneous and rationally suspended, 
\begin{displaymath}
    (-1)^{n+m+2} \kappa_f^{\mathrm{loc}} > 0.
\end{displaymath}
\end{proposition}
\begin{proof}
    Since $\nu_h \preceq_{\mathrm{cx}} \nu_{\IH_{n}}$, by Proposition \ref{prop:convex-order-QH} we have that 
    \begin{displaymath} 
        \nu_f = \nu_g * \nu_h \preceq_{\mathrm{cx}}\nu_g*\nu_{\IH_{n}}.
    \end{displaymath} 
    This means that 
    \begin{displaymath}
        \int_{\RBbb} \phi_{n+m+1} (d (\nu_g * \nu_{\IH_n}) - d \nu_f) \geq 0.
    \end{displaymath}
    Since our measures are invariant under $t\mapsto n+m+2-t$, by equation \eqref{eq:functional-equation-phi}, a lower bound for $(-1)^{n+m+2}\kappa_f^{\mathrm{loc}}/\mu(f)$ is therefore
    \begin{equation}\label{eq:difference-integrals-g-h}
        \frac{1}{2}\int_{\RBbb} \varphi (d (\nu_{\IH_m}\ast\nu_{\IH_n})-d(\nu_g * \nu_{\IH_n}))-\frac{1}{12}
    \end{equation}
where
\begin{displaymath}
    \varphi(t)=\left(t-\frac{m+n+2}{2}\right)^{2}-\widetilde{B}_{2}(t).
\end{displaymath}

The integral of the term $\widetilde{B}_{2}(t)$ is 0. Indeed, by Fubini and the periodicity of $\widetilde{B}_{2}(t)$ one reduces to the computation, for any $a \in \RBbb$:   
\begin{displaymath}
    \int_{a}^{1+a} \widetilde{B}_2(t) dt = \int_0^1 \widetilde{B}_2(t) dt = 0. 
\end{displaymath}
It thus remains to integrate the quadratic term. This is also addressed by Fubini and taking into account \eqref{eq:variance-continuous}. The expression \eqref{eq:difference-integrals-g-h} finally simplifies to
\begin{displaymath}
    \frac{m-1}{24}-\frac{1}{2}V_{g}\geq \frac{\lambda_{\mathrm{\min}}(g)-1}{12},
\end{displaymath}
where the last inequality follows from the Hertling conjecture for $g$.

\end{proof}
\subsubsection{} We now place ourselves in the global geometric setting of the introduction. Hence let $f\colon X\to\DBbb$ be a projective degeneration of Calabi--Yau varieties of dimension $n\geq 3$, with smooth total space $X$, such that $X_0$ is singular and has only isolated terminal singularities. We write $f_x$ for the germ of $f$ around $x \in X_0$. Then 
\begin{displaymath}  \kappa_f=\kappa_{f}^{\mathrm{loc}}=\sum_{x \in X_0^{\operatorname{sing}}} \kappa_{f_x}^{\mathrm{loc}}.
\end{displaymath}
This expression, together with our previous results about $\kappa_{f_{x}}^{\mathrm{loc}}$, can then be summarized in the following theorem.

\begin{theorem}\label{thm:conjectureAcases}
Let the assumptions be as above. Then Conjecture \ref{conjectureintro-kappa-positive} holds if the singularities are of the following form
\begin{enumerate}
    \item Suspensions of isolated plane curve singularities by an even number of squares.
    \item Doubly suspended singularities of quasi-homogeneous singularities. 
    \item Quasi-homogeneous singularities with $\lambda_{\mathrm{\min}}>3/2$.
    \item Rationally suspended quasi-homogeneous singularities.
\end{enumerate}
In particular, Conjecture \ref{conjectureintro-kappa-positive} holds for ADE singularities. 
\end{theorem}

\subsection{The $ADE$ singularities in dimension $n$}\label{subsec:ADEsingulcomp}

In this section, we record the spectra of the classical $ADE$ singularities and the computations of the corresponding local invariants $\kappa_f^{\mathrm{loc}}$. 

Recall that the ADE singularities are for $n=2$ the canonical singularities, and for $n \geq 3$ they are terminal. They are given by the local equations
\[
\begin{array}{cr}
A_k
&
z_0^{k+1}+z_1^2+\cdots+z_n^2=0,
\\[2mm]
D_k
&
z_0^{k-1}+z_0z_1^2+z_2^2+\cdots+z_n^2=0,
\\[2mm]
E_6
&
z_0^4+z_1^3+z_2^2+\cdots+z_n^2=0,
\\[2mm]
E_7
&
z_0^3z_1+z_1^3+z_2^2+\cdots+z_n^2=0,
\\[2mm]
E_8
&
z_0^5+z_1^3+z_2^2+\cdots+z_n^2=0.
\end{array}
\]

\subsubsection{} These are summarized in a table below, together with remarks about when a Calabi--Yau family with such singularities has finite monodromy or not. This latter statement is an application of Lemma \ref{lemma:finitemonodromy} and the explicit description of the spectrum which follows from the description of the spectrum of a quasi-homogeneous singularity in \eqref{eq:qhomspectrum}. The formulas below are valid for $n \geq 2$. When $n \geq 3$, the corresponding singularities are terminal. 
\medskip

\begin{center}
\renewcommand{\arraystretch}{1.5}
\setlength{\extrarowheight}{3pt}
\setlength{\tabcolsep}{2pt}
\begin{tabular}{|c|c|c|c|c|}
\hline
\multirow{2}{*}{\textbf{type}}
&
\multirow{2}{*}{\textbf{spectrum}}
&
\multicolumn{2}{c|}{$\mathbf{\kappa}_f^{\mathrm{loc}}$}
&
\multirow{2}{*}{\shortstack{\textbf{finite monodromy}\\(sufficient conditions)}}
\\
\cline{3-4}
&
&
$n$ even
&
$n$ odd
&
\\
\hline
${A_k}$
&
$\displaystyle \frac{n}{2}+\frac{i}{k+1},
\quad i=1,\ldots,k$
&
$\displaystyle -\frac{k(n-2)}{24}$
&
$\displaystyle \frac{k(n-2)}{24}+c_{A_k}$
&
$k$ or $n$ even
\\[7pt]
\hline
$D_k$
&
$\displaystyle
\frac{n}{2}+\frac{2i-1}{2(k-1)},
\quad i=1,\ldots,k-1,
\qquad \frac{n+1}{2}$
&
$\displaystyle -\frac{k(n-2)}{24}$
&
$\displaystyle \frac{k(n-2)}{24}+c_{D_k}$
&
$n$ even
\\[7pt]
\hline
${E_6}$
&
$\displaystyle
\frac{6n+j}{12},
\quad j\in\{1,4,5,7,8,11\}$
&
$\displaystyle -\frac{n-2}{4}$
&
$\displaystyle \frac{n-2}{4}+c_{E_6}$
&
yes
\\[7pt]
\hline
${E_7}$
&
$\displaystyle
\frac{9n+j}{18},
\quad j\in\{1,5,7,9,11,13,17\}$
&
$\displaystyle -\frac{7(n-2)}{24}$
&
$\displaystyle \frac{7(n-2)}{24}+c_{E_7}$
&
$n$ even
\\[7pt]
\hline
${E_8}$
&
$\displaystyle
\frac{15n+j}{30},
\quad j\in\{1,7,11,13,17,19,23,29\}$
&
$\displaystyle -\frac{n-2}{3}$
&
$\displaystyle \frac{n-2}{3}+c_{E_8}$
&
yes
\\[7pt]
\hline
\end{tabular}
\end{center}
\bigskip

Here the $c$ in the case of odd $n$ are as follows:

\bigskip
\begin{center}
\renewcommand{\arraystretch}{1.5}
\begin{tabular}{|c|c|c|c|c|c|c|c|}
\hline

& $A_k,\ k$ odd
& $A_k,\ k$ even
& $D_k,\ k$ odd
& $D_k,\ k$ even
& $E_6$
& $E_7$
& $E_8$
\\
\hline
& & & & & & &\\ [-16pt]
$c$
& $\displaystyle\frac18$
& $\displaystyle\frac{k}{8(k+1)}$
& $\displaystyle\frac18$
& $\displaystyle\frac{k}{8(k-1)}$
& $\displaystyle\frac1{12}$
& $\displaystyle\frac7{72}$
& $\displaystyle\frac1{15}$
\\[8pt]
\hline
\end{tabular}
\end{center}

\subsubsection{} We don't know if the other cases have finite monodromy or not, and it seems subtle how to lift information from the cohomology of the Milnor fiber to the general fiber. This is a genuinely global question, depending on how the local vanishing cycles map to the cohomology of the nearby fiber. The case $A_1$ when $n$ is odd is unipotent by the Picard--Lefschetz formula. For $n\geq 3$, degenerations with these singularities and finite monodromy give counterexamples to the N\'eron--Ogg--Shafarevich criterion for Calabi--Yau manifolds, as discussed in \textsection \ref{subsection:leitmotif}.

\subsection{The Cynk--van Straten degeneration}

In \cite{CynkStraten} Cynk and van Straten construct a family of Calabi--Yau 3-folds over a punctured unit disc with finite monodromy, but without any smooth filling. It was not the first example in dimension 3, but it was the first one with non-isolated singularities in the special fiber.  

\subsubsection{} We prove that this example is also covered by our formalism, by computing $\kappa_f$ in this setting. In fact, because of Theorem \ref{theorem:nosmoothfilling},  we obtain a stronger statement: the Cynk--van Straten family does not have birational smooth fillings. More precisely, we prove: 
\begin{proposition} \label{prop:cynkstratenekappa}
For the Cynk--van Straten degeneration, one has
\begin{displaymath}
    \kappa_f=\frac12.
\end{displaymath}
\end{proposition}
 For the computation, we recall the construction. We denote the family in \cite{CynkStraten} by $f:Y\to\DBbb$. It arises as the restriction of an algebraic family, it has smooth total space $Y$ and smooth
Calabi--Yau fibers $Y_t$ for $t\neq 0$, satisfying
\begin{equation}\label{eq:cynkstratencohomology}
    h^{1,1}(Y_t)=41,\qquad h^{1,2}(Y_t)=1,
\end{equation}
and with special fiber $Y_0$ singular along a line $L$ (double along $L$ with four pinch points). It has rational singularities. The  blow-up $Z_0\to Y_0$ of $Y_0$ along $L$ is a smooth Calabi--Yau 3-fold.

With these definitions, the expression for $\kappa_f$ in Theorem \ref{thm:kappaBCOVgeneral} takes the form from {\cite[Theorem 7.6 (1)]{cdg2}}. More precisely, for a projective Calabi--Yau degeneration $g\colon X\to\DBbb$, denote by $D_{i}$ the irreducible components of the special fiber, and by $\widetilde{D}_{i}$ the respective normalizations. Denote also by $B$ the divisor of the evaluation map $g^\ast g_\ast K_{X/\DBbb} \to K_{X/\DBbb}$. Then
\begin{equation}\label{eq:kappadim3}
\kappa_g
=
-\frac{1}{6}\bigl(\chi(X_\infty)-\chi(X_0)\bigr)
-\left(\frac{\chi(X_\infty)}{12}
+3\right)\alpha^{0,3}+\alpha^{1,1}-\alpha^{1,2}
-\sum_i \chi\left(\mathcal O_{\widetilde{D}_{i}}\right)
+\frac{1}{12}\int_B c_3(\Omega_X).
\end{equation}

\subsubsection{}
With this in mind, the computation of $\kappa_f$ in \eqref{eq:kappadim3} follows directly from the below considerations. 

\begin{itemize}
\item \emph{The Euler characteristics.}
Since $Y_\infty$ is Calabi--Yau, it follows from \eqref{eq:cynkstratencohomology} that
\begin{displaymath}
    \chi(Y_\infty)=2\bigl(h^{1,1}(Y_\infty)-h^{1,2}(Y_\infty)\bigr)=2(41-1)=80.
\end{displaymath}
Cynk--van Straten compute that $H^k(Y_0,\QBbb)\cong H^k(Y_\infty,\QBbb)$ for $k\neq 3,4$ and show that
there are exact sequences, where $E$ denotes a certain elliptic curve lying above $L$,
\begin{equation}\label{eq:exactsequencesCynkStraten}
\begin{split}
& 0\to H^3(Y_0)\to H^3(Y_\infty)\to H^1(E)(-1)\to 0, \\
& 0\to \bigoplus_{p\in\Sigma}\QBbb(-2)_p \to H^4(Y_0)\to H^4(Y_\infty)\to 0,
\end{split}
\end{equation}
with $|\Sigma|=4$.
From this, one then computes that $\chi(Y_0)=86$, so that $$-\frac16\bigl(\chi(Y_\infty)-\chi(Y_0)\bigr)
=-\frac16(80-86)=1.$$

\medskip\noindent
\item \emph{The terms $\chi(\mathcal O_{\widetilde{D}})$and $B$.}
Here $Y_0$ is irreducible and reduced, so $B=0$, cf. \cite[Section 2.1]{cdg}. Since $Y_0$ has rational singularities, we have $\chi(\Ocal_{Z_0})=\chi(\Ocal_{Y_0})=\chi(\Ocal_{Y_\infty})=0$.

\medskip\noindent
\item \emph{The monodromy terms.}
Since $Y_{0}$ has rational singularities, we have $\alpha^{0,3}=0$ by \cite[Theorem A]{cdg}. Since $H^{2}(Y_{\infty})\simeq H^{2}(Y_{0})$, we deduce $\alpha^{1,1}=0$.

The sequences \eqref{eq:exactsequencesCynkStraten} are coming from a sequence of the type \eqref{eq:mixedlongexactvanishingcycles}, and hence compatible with the monodromy action, which has order $2$ on $H^3(Y_\infty)$. We hence have, with self-explanatory notation: 
\begin{displaymath}
 \alpha^{p,3-p}= \alpha^{p-1, 2-p}_E.
\end{displaymath}
It also follows from \eqref{eq:exactsequencesCynkStraten} that the part where the monodromy acts as $-1$ is given by $H^1(E)(-1)$ and hence $\alpha^{0,1}_E = \alpha^{1,0}_E=1/2.$
We conclude that
\begin{displaymath}
\alpha^{1,2} = \alpha^{2,1}=\frac{1}{2}.
\end{displaymath}
\end{itemize}

\subsection{An example of Fang, Lu and Yoshikawa}\label{subsec:example-FLY}

Let $S$ be an Enriques surface, with universal double covering $X \to S$ with $X$ a K3 surface. Let $\iota$ be the corresponding involution on $X$. If $E$ is a fixed additional elliptic curve, we can consider $X\times E$ modulo the diagonal action $(\iota, -1)$: 
\begin{equation}\label{eq:quotientYoshikawa} X_{S,E}=(X\times E)/\langle(\iota, -1)\rangle.
\end{equation}
In this subsection we show that there are natural families of this type which have finite monodromy, but do not admit any birational smooth fillings.

\subsubsection{} Since $\iota$ acts without fixed points, the quotient is smooth. Since $\iota$ acts by $-1$ on the holomorphic two-form of $X$ and $-1$ acts by $-1$ on the holomorphic one-form of $E$, the diagonal action preserves their product three-form and the quotient is a  3-fold with trivial canonical bundle. This is a Borcea--Voisin Calabi--Yau 3-fold, introduced independently in \cite{Borcea1996} and \cite{Voisin1993}, and this particular version was studied by Ferrara--Harvey--Strominger--Vafa \cite{FHSV}. By \cite[Theorem 13.3]{FLY}, there is a constant $C > 0$ such that 
\begin{equation}\label{eq:BCOV-HSFV-computation}
    \tau_{\mathrm{BCOV}}(X_{S,E})=C\|\Phi(S)\|^2 \| \Delta(E)\|^2.
\end{equation}
Here $\Phi$ is a Borcherds product, and $\Delta$ is the Ramanujan $\Delta$-function, and the norms are Petersson norms. We refer to \cite[Section 13.2]{FLY} for further details. 

\subsubsection{}  

The coarse moduli space of Enriques surfaces is $\Mcal^0= (\Omega \setminus D)/\Gamma$, where $\Omega$ is a type IV period domain, $D$ is a discriminant divisor and $\Gamma$ is a certain arithmetic group. All points in $\Mcal= \Omega/\Gamma$ correspond to K3 surfaces, possibly non-projective. The Borcherds product is defined on $\Mcal$ and vanishes exactly on the discriminant locus \cite{Borcherds-Enriques}.

\begin{proposition}\label{prop:tau-bcov-BV}
    Let $\mathcal{S}\to\DBbb^\times$ be a family of Enriques surfaces such that the moduli map extends to a map $\gamma: \DBbb \to \Mcal$ with $\gamma(0)\in D/\Gamma$. Then the corresponding family of Borcea--Voisin 3-folds satisfies 
    \begin{displaymath}
        \log \tau_{\mathrm{BCOV}}(X_{\mathcal{S}_t,E})=\kappa \log|t|^2+O(1) 
    \end{displaymath}
    with $\kappa > 0.$
\end{proposition}
\begin{proof}
The only varying factor in
\eqref{eq:BCOV-HSFV-computation} is $\|\Phi\|^2$.
Since $\Mcal$ is locally a quotient by a finite group,
after shrinking the disc and making a finite base change
$t=s^m$, the extended moduli map admits a holomorphic lift
$\widetilde{\gamma}\colon\DBbb_s\to\Omega$, with $\widetilde{\gamma}(0)\in D$.

In a local trivialization of the automorphic line bundle,
write $\phi$ for the holomorphic representative of $\Phi$.
Along the lifted period map,
\[
    \|\Phi(\widetilde{\gamma}(s))\|^2
    =h(s)|\phi(\widetilde{\gamma}(s))|^2,
\]
where $h$ is smooth and positive.
Since $\Phi$ vanishes precisely on $D$, and
$\widetilde{\gamma}(s)\notin D$ for $s\neq0$, we may write
\[
    \phi(\widetilde{\gamma}(s))=s^r u(s),
    \qquad r\geq1,\quad u(0)\neq0.
\]
Consequently,
\[
    \log\tau_{\mathrm{BCOV}}(X_{\mathcal{S}_{s^m},E})
    =r\log|s|^2+O(1)
    =\frac{r}{m}\log|t|^2+O(1),
\]
which proves the assertion.
\end{proof}
\subsubsection{}  
We provide a concrete family with finite monodromy. Specializing the statement of \cite[Theorem 2.6]{Yoshikawak31} to the Enriques case, one finds that a map $\DBbb \to \Mcal$ intersecting the discriminant locus transversally gives rise to  a projective degeneration $\Zcal \to \DBbb$ of K3 surfaces $\Zcal_t, t \neq 0$ with fixed-point-free involution, such that $\Zcal_0$ has a unique ordinary double point, and the involution extends to the central fiber and acts without fixed points outside the singular point. This naturally provides us with a family $\mathcal{S}\to \DBbb^\times$ to which Proposition \ref{prop:tau-bcov-BV} applies. From the Picard--Lefschetz theorem, it follows that both the corresponding Enriques and Borcea--Voisin families have finite monodromy. The Borcea--Voisin family has a one-dimensional singular locus given by $(\Zcal_0^{\sing} \times E)/(\iota, -1)\simeq \PBbb^1$.

\begin{corollary}
    There are Borcea--Voisin families with finite monodromy, degenerating with a 1-dimensional singular locus, but which do not admit any birational smooth fillings.
\end{corollary}
\begin{proof}
The claim follows from the above discussion and Proposition \ref{prop:criterion-smooth-filling}. In order to be able to apply the latter, we need to recall that the Borcea--Voisin construction yields projective Calabi--Yau 3-folds, since they are Kähler with $h^{2,0}=0$.
\end{proof}

\subsection{Brieskorn--Pham singularities}

In this section we study the case of Brieskorn--Pham singularities. Namely: 

\begin{theorem}\label{thm:localkappaBP}
Let \begin{displaymath}
    f=z_0^{a_0} + \ldots + z_n^{a_n}
\end{displaymath} be a terminal Brieskorn--Pham singularity with $n\geq 3$. Then Conjecture \hyperref[conj:localkappa]{$\mathrm{A}^{\prime}$} holds. In particular, Conjecture \ref{conjectureintro-kappa-positive} holds for Calabi--Yau degenerations with such singularities.
\end{theorem}

\subsubsection{} We first provide a preliminary discussion on terminal Brieskorn--Pham singularities and the problem to solve. 
Recall that the minimal spectral number is given, in this case, by
\begin{displaymath}
    \lambda_{\mathrm{\min}}=\sum_{i}\frac{1}{a_{i}}.
\end{displaymath}
From \cite[Section 4.1]{Weys}, we know that the terminal condition entails, but is not equivalent to, 
\begin{displaymath}
    \sum \frac{1}{a_i} > 1 + \frac{1}{L},\quad L=\mathrm{lcm}(a_{0},\ldots,a_{n}).
\end{displaymath}
Therefore, we have
\begin{equation}\label{def:quantity-M}
    M:=L(\lambda_{\mathrm{\min}}-1)\geq 2.
\end{equation}

\subsubsection{} As for the problem at hand, since the Hertling conjecture holds for Brieskorn--Pham singularities, we can write
\begin{displaymath}
    (-1)^{n+1}\frac{\kappa_{f}^{\mathrm{loc}}}{\mu}=\frac{\lambda_{\mathrm{\min}}-1+6\beta_{f}}{12}.
\end{displaymath}
Hence, after multiplication by $L$, we must prove
\begin{displaymath}
    M+6L\beta_{f}>0.
\end{displaymath}
Since $M\geq 2$, we note that this would follow if we could prove
\begin{equation}\label{eq:bound-for-beta}
    |\beta_{f}|\leq\frac{1}{6L}.
\end{equation}
The proof, presented below,  distinguishes two classes of Brieskorn--Pham singularities, according to the structure of the exponents. The first argument is of a general nature, and relies on the case of rationally suspended singularities. The argument in the second scenario resolves \eqref{eq:bound-for-beta} in the rest of the cases.

\subsubsection{} The first case assumes that there exists an $a_j$ such that $\sum_{i \neq j} \frac{1}{a_i} > 1$. In this scenario, the singularity 
\begin{displaymath}
    \sum_{i\neq j} z_i^{a_i}=0
\end{displaymath}
is rational, so the original singularity is rationally suspended. The statement then follows from Proposition \ref{prop:rationally-suspended}. 

In the second case, all such partial sums satisfy $\sum_{i\neq j} \frac{1}{a_i} \leq 1.$ An elementary observation is that at most one exponent can be 2, since the number of variables is at least 4. To treat this, we will formulate below a combinatorial proposition, Proposition \ref{proposition:combinatoriallemma}, which has as a special case \eqref{eq:bound-for-beta} whenever at most one exponent is 2.

\subsubsection{} 

In preparation for the proposition, consider a family, possibly with repeated elements, of pairs $$D=\lbrace (N_1, b_1),\ldots, (N_{r},b_{r})\rbrace$$ of positive integers, with $b_i | N_i$ and $N_i \geq 2$.  We will generalize the definition of $\beta_{f}$ to such $D$, in a way that $\beta_{f}$ corresponds to the family  with $N_{i}=b_{i}=a_{i}$.

For a tuple $(N, b)$ as above, we define an operator on 1-periodic real functions
\begin{displaymath}
    T_{N,b} = \frac{N E_b - I}{N-1},
\end{displaymath}
where $E_b$ is the averaging operator 
\begin{displaymath}
    E_b \varphi(x) = \frac{1}{b} \sum_{j=0}^{b-1} \varphi\left( x + \frac{j}{b} \right).    
\end{displaymath}
We note that for tuples $(N,b)$ and $(N',b')$, the corresponding $T$ operators commute:
\begin{displaymath}
    T_{N,b}T_{N',b'}=T_{N',b'}T_{N,b}.
\end{displaymath}

With the notation above, we define 
\begin{displaymath}
    \beta_{D}(x) = T_{N_1, b_1} \cdots T_{N_r, b_r} \widetilde{B}_{2}(x),\quad\beta_{D}=\beta_{D}(0).
\end{displaymath}
For the family $D_{f}=\lbrace (a_{0},a_{0}),\ldots, (a_{n},a_{n})\rbrace$, it is elementary to check that $\beta_{D_{f}}$ recovers $\beta_{f}$. 

\subsubsection{} The proposition takes the following form:

\begin{proposition}\label{proposition:combinatoriallemma}
    Let $D$ be a tuple as above. Then, if $(2,2)$ appears at most once in $D$, we have 
    \begin{displaymath}
        |\beta_{D}| \leq \frac{1}{6L(D)},\quad L(D)=\mathrm{lcm} (b_1, \ldots, b_r).
    \end{displaymath}
\end{proposition}

\begin{proof}
    We perform an induction argument on $r$, with the auxiliary definition  \begin{equation}\label{eq:betaempty}\beta_\emptyset=\widetilde{B}_2(0)=\frac{1}{6}
    \end{equation} when $r=0$.

    Let $(N,b)$ be a new pair, and let $D'$ be obtained by appending $(N,b)$ to the family $D$. 
    By definition, 
    \begin{equation}\label{eq:betainitial}
        \beta_{D'}(x) = T_{N_1, b_1} \ldots T_{N_r, b_r} T_{N,b} \widetilde{B}_{2} (x) = \frac{N T_{N_1, b_1} \ldots T_{N_r, b_r} E_b \widetilde{B}_2(x) - \beta_{D}(x)}{N-1}.
    \end{equation}
    By the Kubert identity/duplication formula \cite[Example 2.1]{ChakrabortyKanemitsuKuzumaki2018} we have 
    $$E_b \widetilde{B}_2(x) = \frac{1}{b^2}\widetilde{B}_2(bx).$$ If we denote by $S_b$ the operator $S_b f(x) = f(bx)$ we hence find that \eqref{eq:betainitial} takes the form 
    \begin{equation}\label{eq:grosse-fraction-bernouillis}
        \beta_{D'}(x)=\frac{\frac{N}{b^2} T_{N_1, b_1} \ldots T_{N_r, b_r} S_b \widetilde{B}_2(x) - \beta_{D}(x)}{N-1}.
    \end{equation}
An elementary argument shows that on 1-periodic functions, we have the relation 
\begin{displaymath}
    T_{N_i, b_i} S_b = S_b T_{N_i, b_i^{\ast}},\quad b_{i}^\ast=b_i/\mathrm{gcd}(b, b_i).
\end{displaymath}
We deduce
\begin{displaymath}
    T_{N_1, b_1} \ldots T_{N_r, b_r} S_b \widetilde{B}_2(x) = S_b T_{N_1, b_1^\ast} \ldots T_{N_r, b_r^\ast} \widetilde{B}_2(x) = \beta_{D^\ast}(bx)
\end{displaymath} 
where $D^\ast$ consists of the $2$-tuples $\lbrace (N_1, b_1^\ast), \ldots, (N_r, b_r^\ast)\rbrace$.

In particular we find, by evaluating at $x=0$,  that
\begin{equation}\label{eq:beta'later}
    \beta_{D'} = \frac{\frac{N}{b^2} \beta_{D^\ast}-\beta_D}{N-1}.
\end{equation}

For the singleton case, take $D=\emptyset$ and $D'=\{(N,b)\}$. Then \eqref{eq:betaempty} applies, and the expression \eqref{eq:beta'later} also holds and collapses to
\begin{displaymath}
    \beta_{(N,b)}=\frac{1}{6}\frac{\frac{N}{b^2}-1}{N-1}=\frac{1}{6b}\frac{k-b}{kb-1}.
\end{displaymath}
where we decomposed $N=kb$. Note that $N\geq 2$ by assumption, and it follows that
\begin{equation}\label{eq:simple-estimate}
    |\beta_{(N,b)}|=\frac{1}{6b}\frac{|k-b|}{kb-1} \leq \frac{1}{6b}.
\end{equation}  
In this step, we allow the case $(2,2)$. Since the $T$ operators commute with each other and we allow at most one factor $(2,2)$, for the case $r\geq 2$ we may assume that $(N,b) \neq (2,2)$.

We next introduce
\begin{displaymath}
    \begin{split}
        &b_{i}=b_{i}^{\ast}h_{i},\quad h_{i}=\mathrm{gcd}(b_{i},b)\\
        &h=\mathrm{lcm}(h_{1},\ldots,h_{r}).
    \end{split}
\end{displaymath}
If $h=1$, then $D=D^\ast$ and $L(D^{\prime})=L(D)b$. In this case, \eqref{eq:beta'later} simplifies to
\begin{displaymath}
    \beta_{D'}=\frac{\frac{N}{b^2} -1}{N-1}\beta_{D}=\frac{\frac{k}{b}-1}{kb-1}\beta_{D}.
\end{displaymath}
Applying the induction assumption to $\beta_D$ and the estimate as in \eqref{eq:simple-estimate}, we obtain
\begin{displaymath}
    |\beta_{D'}|\leq\frac{1}{b}\frac{1}{6L(D)}=\frac{1}{6L(D')}.
\end{displaymath}
If $h\geq 2$, again by the induction assumption on $\beta_{D}$ and $\beta_{D^*}$, we have
\begin{equation}\label{eq:bound-beta-induction}
    |\beta_{D'}| \leq \frac{1}{6} \frac{\frac{N}{b^2L(D^*)}  +\frac{1}{L(D)}}{N-1}
\end{equation}
Combined with the identities
\begin{displaymath}
    L(D)=L(D^\ast)h,\quad L(D')=\frac{b}{h}L(D),
\end{displaymath}
equation \eqref{eq:bound-beta-induction} yields
\begin{displaymath}
    |\beta_{D'}|\leq \frac{1}{6L(D')}\frac{k+
    \frac{b}{h}}{kb-1}.
\end{displaymath}
Because $h\geq 2$, we have 
\begin{displaymath}
    kb-1  = k \cdot h \cdot \frac{b}{h} -1 \geq 2 k \cdot \frac{b}{h} -1 \geq k + \frac{b}{h}, 
\end{displaymath}
as long as $b/h \geq 2$ and $k \geq 1$. If $b=h$, the required inequality is $k(b-1)\geq 2$. This holds for $k\geq 2$, since $b=h\geq 2$, and for $k=1$ unless $b=2$. The only excluded case is thus $(N,b)=(2,2)$.
\end{proof}

\subsection{Beyond terminal singularities} 

It is natural to ask if the sign of $\kappa_{f}^{\mathrm{loc}}$ can more generally be controlled according to the singularity type. 

\subsubsection{} A refinement of Proposition \ref{proposition:combinatoriallemma} in fact gives that $(-1)^{n+1}\kappa_f^{\mathrm{loc}} \geq 0$ for a Brieskorn--Pham Du Bois singularity, and $(-1)^{n+1}\kappa_f^{\mathrm{loc}} \leq 0$ for non-Du Bois singularities. Vanishing occurs, but is rare. A necessary condition for vanishing is 
\begin{displaymath}
    \sum \frac{1}{a_i} = 1\quad\text{or}\quad 1 \pm \frac{1}{\mathrm{lcm}(a_0, \ldots, a_n)}.
\end{displaymath}
We don't provide the details of this fact.

Below we record a few examples of singularities with $\kappa_{f}^{\mathrm{loc}}=0$, found by computer. 

\begin{description}

\item[Dimension $3$]
Among Brieskorn--Pham singularities with $2\leq a_i\leq 20$, the only examples are
\begin{displaymath}
        (3,3,4,13),
        \qquad
        (3,3,5,8).
\end{displaymath}

\item[Dimension $4$]
Among exponent vectors with $2\leq a_i\leq 15$, the only examples found are
\begin{displaymath}
        (3,3,7,10,11),
        \qquad
        (3,4,4,10,15).
\end{displaymath}

Outside this range, experimentation with individual examples revealed vanishing examples at larger exponents.  For instance, the Sylvester type sequences provide examples:
\begin{displaymath}
        (2,3,7,43,1805),
        \qquad
        (2,3,7,43,1806).
\end{displaymath}

\item[Dimension $5$]
In the small ranges tested, with $2 \leq a_i \leq 15$, no additional vanishing examples were found.

\end{description}
\subsubsection{} The above computations suggest that $(-1)^{n+1} \kappa_f^{\mathrm{loc}} \geq 0$ for more general canonical or Du Bois singularities. However, for the quasi-homogeneous singularity 
\begin{displaymath}
    f=z_{0}^2 z_1 + z_1^2 z_2 + z_2^5  z_3 + z_3^{18} z_4 + z_4^{18}=0
\end{displaymath}
with weights $\left(
\frac{1927}{6480},
\frac{1313}{3240},
\frac{307}{1620},
\frac{17}{324},
\frac{1}{18}
\right)$ one computes that
\begin{displaymath}
    \kappa_f^{\mathrm{loc}} = \frac{11}{540}.
\end{displaymath}
However, the singularity is canonical and the naive expectation from above is that it should have been non-positive. There are many such examples, generated by a computer search.\bigskip

\noindent\textbf{AI statement.} Some of the ideas leading to Proposition \ref{prop:convex-order-QH} and Proposition \ref{proposition:combinatoriallemma} were suggested in discussions with ChatGPT (GPT-5.6 Sol). The arguments were subsequently developed, verified, and put into their final form by the authors. 

\bibliographystyle{amsplain}
\bibliography{Spectralgenus} 

\providecommand{\bysame}{\leavevmode\hbox to3em{\hrulefill}\thinspace}
\providecommand{\MR}{\relax\ifhmode\unskip\space\fi MR }
\providecommand{\MRhref}[2]{%
  \href{http://www.ams.org/mathscinet-getitem?mr=#1}{#2}
}
\providecommand{\href}[2]{#2}
\begin{thebibliography}{10}

\bibitem{Almiron-Schulze}
P.~Almir\'{o}n and M.~Schulze, \emph{Limit spectral distribution for
  non-degenerate hypersurface singularities}, C. R. Math. Acad. Sci. Paris
  \textbf{360} (2022), 699--710.

\bibitem{bcov}
M.~Bershadsky, S.~Cecotti, H.~Ooguri, and C.~Vafa, \emph{Kodaira-{S}pencer
  theory of gravity and exact results for quantum string amplitudes}, Comm.
  Math. Phys. \textbf{165} (1994), no.~2, 311--427.

\bibitem{Billingsley}
P.~Billingsley, \emph{Convergence of probability measures}, 2nd ed., Wiley Ser.
  Probab. Stat., Wiley, 1999.

\bibitem{bismutbost}
J.-M. Bismut and J.-B. Bost, \emph{Fibr\'es d\'eterminants, m\'etriques de
  {Q}uillen et d\'eg\'en\'erescence des courbes}, Acta Math. \textbf{165}
  (1990), no.~1-2, 1--103.

\bibitem{Bismut--Freed}
J.-M. Bismut and D.~Freed, \emph{The analysis of elliptic families. {I}.
  {M}etrics and connections on determinant bundles}, Comm. Math. Phys.
  \textbf{106} (1986), no.~1, 159--176.

\bibitem{BGS1}
J.-M. Bismut, H.~Gillet, and C.~Soul{\'e}, \emph{Analytic torsion and
  holomorphic determinant bundles. {I}. {B}ott-{C}hern forms and analytic
  torsion}, Comm. Math. Phys. \textbf{115} (1988), no.~1, 49--78.

\bibitem{BGS3}
\bysame, \emph{Analytic torsion and holomorphic determinant bundles. {III}.
  {Q}uillen metrics on holomorphic determinants}, Comm. Math. Phys.
  \textbf{115} (1988), no.~2, 301--351.

\bibitem{Borcea1996}
C.~Borcea, \emph{K3 surfaces with involution and mirror pairs of {C}alabi-{Y}au
  manifolds}, Mirror Symmetry II, AMS/IP Studies in Advanced Mathematics,
  vol.~1, American Mathematical Society and International Press, 1996,
  pp.~717--743.

\bibitem{Borcherds-Enriques}
R.~E. Borcherds, \emph{The moduli space of {E}nriques surfaces and the fake
  {M}onster {L}ie superalgebra}, Topology \textbf{35} (1996), no.~3, 699--710.

\bibitem{Boucksom:minimal-models-degenerations}
S.~Boucksom, \emph{Remarks on minimal models of degenerations}, Unpublished
  note, 2014.

\bibitem{Brieskorn}
E.~Brieskorn, \emph{Die {M}onodromie der isolierten {S}ingularit\"{a}ten von
  {H}yperfl\"{a}chen}, Manuscripta Math. \textbf{2} (1970), 103--161.

\bibitem{ChakrabortyKanemitsuKuzumaki2018}
K.~Chakraborty, S.~Kanemitsu, and T.~Kuzumaki, \emph{Seeing the invisible:
  Around generalized {K}ubert functions}, Annales Universitatis Scientiarum
  Budapestinensis de Rolando E{\"o}tv{\"o}s Nominatae. Sectio Computatorica
  \textbf{47} (2018), 185--195.

\bibitem{CynkStraten}
S.~Cynk and D.~van Straten, \emph{A special {C}alabi-{Y}au degeneration with
  trivial monodromy}, Commun. Contemp. Math. \textbf{24} (2022), no.~8, Paper
  No. 2150055, 15.

\bibitem{DimcaMilnorweightedhomo}
A.~Dimca, \emph{On the {M}ilnor fibrations of weighted homogeneous
  polynomials}, Compositio Math. \textbf{76} (1990), no.~1-2, 19--47.

\bibitem{Dimca:sheaves-in-topology}
\bysame, \emph{Sheaves in topology}, Universitext, Springer-Verlag, Berlin,
  2004.

\bibitem{Dirks-Mustata}
B.~Dirks and M.~Musta{\c{t}}{\u{a}}, \emph{Minimal exponents of hyperplane
  sections: a conjecture of {Teissier}}, J. Eur. Math. Soc. (JEMS) \textbf{25}
  (2023), no.~12, 4813--4840.

\bibitem{duBois}
P.~Du~Bois, \emph{Complexe de de {Rham} filtr{\'e} d'une vari{\'e}t{\'e}
  singuli{\`e}re}, Bull. Soc. Math. Fr. \textbf{109} (1981), 41--81.

\bibitem{DURF1}
D.~Eriksson and G.~Freixas~i Montplet, \emph{The spectral genus of an isolated
  hypersurface singularity and its relation to the {Milnor} number and analytic
  torsion}, Doc. Math. \textbf{31} (2026), no.~6, 1387--1419.

\bibitem{cdg}
D.~Eriksson, G.~Freixas~i Montplet, and C.~Mourougane, \emph{Singularities of
  metrics on {H}odge bundles and their topological invariants}, Algebraic
  Geometry \textbf{5} (2018), 1--34.

\bibitem{cdg2}
\bysame, \emph{B{COV} invariants of {C}alabi-{Y}au manifolds and degenerations
  of {H}odge structures}, Duke Math. J. \textbf{170} (2021), no.~3, 379--454.

\bibitem{cdg3}
\bysame, \emph{On genus one mirror symmetry in higher dimensions and the {BCOV}
  conjectures}, Forum Math. Pi \textbf{10} (2022), Paper No. e19, 53.

\bibitem{FLY}
H.~Fang, Z.~Lu, and K.-I. Yoshikawa, \emph{Analytic torsion for {C}alabi-{Y}au
  threefolds}, J. Differential Geom. \textbf{80} (2008), no.~2, 175--259.

\bibitem{FHSV}
S.~Ferrara, J.~Harvey, A.~Strominger, and C.~Vafa, \emph{Second-quantized
  mirror symmetry}, Phys. Lett. B \textbf{361} (1995), no.~1-4, 59--65.

\bibitem{Friedman:degenerating-family}
R.~Friedman, \emph{A degenerating family of quintic surfaces with trivial
  monodromy}, Duke Math. J. \textbf{50} (1983), 203--214.

\bibitem{Friedmansimthreefolddouble}
\bysame, \emph{Simultaneous resolution of threefold double points}, Math. Ann.
  \textbf{274} (1986), no.~4, 671--689.

\bibitem{Friedman-Laza-isolated}
R.~Friedman and R.~Laza, \emph{The higher {Du} {Bois} and higher rational
  properties for isolated singularities}, J. Algebr. Geom. \textbf{33} (2024),
  no.~3, 493--520.

\bibitem{Fristedt-Gray}
B.~Fristedt and L.~Gray, \emph{A modern approach to probability theory},
  Birkh{\"a}user, 1997.

\bibitem{ZhangFu-birationalBCOV}
L.~Fu and Y.~Zhang, \emph{Motivic integration and birational invariance of
  {BCOV} invariants}, Selecta Math. (N.S.) \textbf{29} (2023), no.~2, Paper No.
  25, 41.

\bibitem{Fulton}
W.~Fulton, \emph{Intersection theory}, second ed., Ergebnisse der Mathematik
  und ihrer Grenzgebiete. 3. Folge. A Series of Modern Surveys in Mathematics,
  vol.~2, Springer-Verlag, Berlin, 1998.

\bibitem{Hertling}
C.~Hertling, \emph{Frobenius manifolds and variance of the spectral numbers},
  New developments in singularity theory ({C}ambridge, 2000), NATO Sci. Ser. II
  Math. Phys. Chem., vol.~21, Kluwer Acad. Publ., Dordrecht, 2001,
  pp.~235--255.

\bibitem{Hironaka:counterexample}
H.~Hironaka, \emph{An example of a non-{K{\"a}hlerian} complex-analytic
  deformation of {K{\"a}hlerian} complex structures}, Ann. Math. (2)
  \textbf{75} (1962), 190--208.

\bibitem{HKTYci}
S.~Hosono, A.~Klemm, S.~Theisen, and S.-T. Yau, \emph{Mirror symmetry, mirror
  map and applications to complete intersection {C}alabi--{Y}au spaces},
  Nuclear Physics B \textbf{433} (1995), no.~3, 501--552.

\bibitem{HKQ}
M.~Huang, A.~Klemm, and S.~Quackenbush, \emph{Topological string theory on
  compact {C}alabi--{Y}au: Modularity and boundary conditions}, Homological
  Mirror Symmetry, Lecture Notes in Physics, vol. 757, Springer, 2009,
  pp.~45--102.

\bibitem{Kato-Saito}
K.~Kato and T.~Saito, \emph{On the conductor formula of {B}loch}, Publ. Math.
  Inst. Hautes \'Etudes Sci. (2004), no.~100, 5--151.

\bibitem{KnudsenMumford}
F.~Knudsen and D.~Mumford, \emph{The projectivity of the moduli space of stable
  curves. {I}. {P}reliminaries on ``det'' and ``{D}iv''}, Math. Scand.
  \textbf{39} (1976), no.~1, 19--55.

\bibitem{Kodaira-Spencer}
K.~Kodaira and D.~C. Spencer, \emph{On deformations of complex analytic
  structures. {III}. {S}tability theorems for complex structures}, Ann. of
  Math. (2) \textbf{71} (1960), 43--76.

\bibitem{Kollar2}
J.~Koll{\'a}r, \emph{Higher direct images of dualizing sheaves. {II}}, Ann. of
  Math. (2) \textbf{124} (1986), no.~1, 171--202.

\bibitem{kollar:singpairs}
\bysame, \emph{Singularities of pairs}, Algebraic geometry---{S}anta {C}ruz
  1995, Proc. Sympos. Pure Math., vol.~62, Amer. Math. Soc., Providence, RI,
  1997, pp.~221--287.

\bibitem{Kollar:Seshadri}
\bysame, \emph{Seshadri's criterion and openness of projectivity}, Proc. Indian
  Acad. Sci., Math. Sci. \textbf{132} (2022), no.~2, 12, Id/No 40.

\bibitem{KLSV-HK}
J.~Koll{\'a}r, R.~Laza, G.~Sacc{\`a}, and C.~Voisin, \emph{Remarks on
  degenerations of hyper-{K{\"a}hler} manifolds}, Ann. Inst. Fourier
  \textbf{68} (2018), no.~7, 2837--2882.

\bibitem{NicaiseLunardon}
L.~Lunardon and J.~Nicaise, \emph{Birational invariance of motivic zeta
  functions of {$K$}-trivial varieties, and obstructions to smooth fillings},
  Preprint, {arXiv}:2401.17772 [math.{AG}], 2024.

\bibitem{FreeMori}
A.~Moriwaki, \emph{Torsion freeness of higher direct images of canonical
  bundles}, Math. Ann. \textbf{276} (1987), no.~3, 385--398.

\bibitem{Navarro}
V.~Navarro~Aznar, \emph{Sur la th\'{e}orie de {H}odge-{D}eligne}, Invent. Math.
  \textbf{90} (1987), no.~1, 11--76.

\bibitem{Peters-Steenbrink}
C.~Peters and J.~Steenbrink, \emph{Mixed {H}odge structures}, Ergebnisse der
  Mathematik und ihrer Grenzgebiete. 3. Folge. A Series of Modern Surveys in
  Mathematics, vol.~52, Springer-Verlag, Berlin, 2008.

\bibitem{Eulerian}
T.~K. Petersen, \emph{Eulerian numbers}, Birkh\"auser Advanced Texts: Basler
  Lehrb\"ucher. [Birkh\"auser Advanced Texts: Basel Textbooks],
  Birkh\"auser/Springer, New York, 2015, With a foreword by Richard Stanley.

\bibitem{Prokhorov1999Complements}
Y.~Prokhorov, \emph{Lectures on complements on log surfaces}, MSJ Memoirs,
  vol.~10, Mathematical Society of Japan, Tokyo, 2001.

\bibitem{Sabbah-Schnell}
C.~{Sabbah} and C.~{Schnell}, \emph{{Degenerating complex variations of Hodge
  structure in dimension one}}, arXiv e-prints (2022), arXiv:2206.08166.

\bibitem{KSaito:distribution}
K.~Saito, \emph{The zeroes of characteristic function {$\chi_f$} for the
  exponents of a hypersurface isolated singular point}, Algebraic varieties and
  analytic varieties ({T}okyo, 1981), Adv. Stud. Pure Math., vol.~1,
  North-Holland, Amsterdam, 1983, pp.~195--217.

\bibitem{MSaitogenus}
M.~Saito, \emph{On the exponents and the geometric genus of an isolated
  hypersurface singularity}, Singularities, {P}art 2 ({A}rcata, {C}alif.,
  1981), Proc. Sympos. Pure Math., vol.~40, Amer. Math. Soc., Providence, RI,
  1983, pp.~465--472.

\bibitem{MSaito:MHP}
\bysame, \emph{Modules de {Hodge} polarisables}, Publ. Res. Inst. Math. Sci.
  \textbf{24} (1988), no.~6, 849--995.

\bibitem{MSaito:MHM}
\bysame, \emph{Mixed {Hodge} modules}, Publ. Res. Inst. Math. Sci. \textbf{26}
  (1990), no.~2, 221--333.

\bibitem{Saito:du-Bois}
\bysame, \emph{On the {Hodge} filtration of {Hodge} modules}, Mosc. Math. J.
  \textbf{9} (2009), no.~1, 151--181.

\bibitem{Weys}
J.~Schepers and W.~Veys, \emph{Stringy {$E$}-functions of hypersurfaces and of
  {B}rieskorn singularities}, Adv. Geom. \textbf{9} (2009), no.~2, 199--217.

\bibitem{Scherk-Steenbrink}
J.~Scherk and J.~H.~M. Steenbrink, \emph{On the mixed {Hodge} structure on the
  cohomology of the {Milnor} fibre}, Math. Ann. \textbf{271} (1985), 641--665.

\bibitem{schmid}
W.~Schmid, \emph{Variation of {H}odge structure: the singularities of the
  period mapping}, Invent. Math. \textbf{22} (1973), 211--319.

\bibitem{Thom-Sebastiani}
M.~Sebastiani and R.~Thom, \emph{Un r{\'e}sultat sur la monodromie. ({A} result
  on the monodromy)}, Invent. Math. \textbf{13} (1971), 90--96.

\bibitem{Stochastic-orders}
M.~Shaked and J.~G. Shantikumar, \emph{Stochastic orders}, Springer Ser. Stat.,
  New York, NY: Springer, 2007.

\bibitem{Soule:lectures}
C.~Soul{\'e}, D.~Abramovich, J.-F. Burnol, and J.~Kramer, \emph{Lectures on
  {Arakelov} geometry}, Camb. Stud. Adv. Math., vol.~33, Cambridge: Cambridge
  University Press, 1992.

\bibitem{Steenbrink-limits}
J.~Steenbrink, \emph{Limits of {H}odge structures}, Invent. Math. \textbf{31}
  (1975/76), no.~3, 229--257.

\bibitem{Steenbrink-mixedonvanishing}
J.~H.~M. Steenbrink, \emph{Mixed {H}odge structure on the vanishing
  cohomology}, Real and complex singularities ({P}roc. {N}inth {N}ordic
  {S}ummer {S}chool/{NAVF} {S}ympos. {M}ath., {O}slo, 1976), Sijthoff and
  Noordhoff, Alphen aan den Rijn, 1977, pp.~525--563.

\bibitem{SteenbrinkMixedAssociated}
\bysame, \emph{Mixed {H}odge structures associated with isolated
  singularities}, Singularities, {P}art 2 ({A}rcata, {C}alif., 1981), Proc.
  Sympos. Pure Math., vol.~40, Amer. Math. Soc., Providence, RI, 1983,
  pp.~513--536.

\bibitem{SteenbrinkVanishingThm}
\bysame, \emph{Vanishing theorems on singular spaces}, Ast\'erisque (1985),
  no.~130, 330--341, Differential systems and singularities (Luminy, 1983).

\bibitem{Steenbrink:Du-Bois}
\bysame, \emph{Du {B}ois invariants of isolated complete intersection
  singularities}, Ann. Inst. Fourier (Grenoble) \textbf{47} (1997), no.~5,
  1367--1377.

\bibitem{TranNoeuds}
L\^e~Dung Tr\'ang, \emph{Sur les noeuds alg\'ebriques}, Compositio Math.
  \textbf{25} (1972), 281--321.

\bibitem{Varchenko-asymptotic}
A.~N. Var\v{c}enko, \emph{Asymptotic {H}odge structure on vanishing
  cohomology}, Izv. Akad. Nauk SSSR Ser. Mat. \textbf{45} (1981), no.~3,
  540--591, 688.

\bibitem{Voisin:fillings}
C.~Voisin, \emph{D{\'e}g{\'e}n{\'e}rations de {Lefschetz} et variations de
  structures de {Hodge}. ({Lefschetz} degenerations and variations of {Hodge}
  structure)}, J. Differ. Geom. \textbf{31} (1990), no.~2, 527--534.

\bibitem{Voisin1993}
\bysame, \emph{Miroirs et involutions sur les surfaces {K3}}, Journ{\'e}es de
  g{\'e}om{\'e}trie alg{\'e}brique d'Orsay - Juillet 1992, Ast{\'e}risque, no.
  218, Soci{\'e}t{\'e} math{\'e}matique de France, 1993, pp.~273--323.

\bibitem{wang}
C.-L. Wang, \emph{On the incompleteness of the {W}eil-{P}etersson metric along
  degenerations of {C}alabi-{Y}au manifolds}, Math. Res. Lett. \textbf{4}
  (1997), no.~1, 157--171.

\bibitem{Yau-Zhang}
S.~T. Yau and L.~Zhang, \emph{An upper estimate of integral points in real
  simplices with an application to singularity theory}, Math. Res. Lett.
  \textbf{13} (2006), no.~5-6, 911--921.

\bibitem{yoshikawa2}
K.-I. Yoshikawa, \emph{Smoothing of isolated hypersurface singularities and
  {Q}uillen metrics}, Asian J. Math. \textbf{2} (1998), no.~2, 325--344.

\bibitem{Yoshikawak31}
\bysame, \emph{{$K3$} surfaces with involution, equivariant analytic torsion,
  and automorphic forms on the moduli space}, Invent. Math. \textbf{156}
  (2004), no.~1, 53--117.

\bibitem{yoshikawa}
\bysame, \emph{On the singularity of {Q}uillen metrics}, Math. Ann.
  \textbf{337} (2007), no.~1, 61--89.

\bibitem{Zhang-Yuguang-Survey}
Y.~Zhang, \emph{Degeneration of {Ricci}-flat {Calabi}-{Yau} manifolds and its
  applications}, Uniformization, Riemann-Hilbert correspondence, Calabi-Yau
  manifolds and Picard-Fuchs equations. Based on the conference, Institute
  Mittag-Leffler, Stockholm, Sweden, July 13--18, 2015, Somerville, MA:
  International Press; Beijing: Higher Education Press, 2018, pp.~551--592.

\end{thebibliography}

\end{document}